%% file: ms.tex
\documentclass[11pt,reqno]{amsart}
\usepackage{fullpage,times,amssymb,amsmath,psfrag,xcolor}
\usepackage{natbib}
\usepackage[colorlinks,citecolor=bbluegray,linkcolor=ddarkbrown,urlcolor=blue,breaklinks]{hyperref}
\usepackage{bbding}
\usepackage{placeins}
\usepackage{algorithmic,algorithm}

\usepackage{graphicx}
\usepackage[off]{auto-pst-pdf} 

 \usepackage{comment}
\usepackage{amsaddr}

\input defs.tex

\let \oldsection \section
\renewcommand{\section}{\vspace{3ex plus 1ex}\oldsection}

\newcommand{\pnm}{\mathcal{P}_{[N-1]}^{(1)}}
\newcommand{\pnp}{\mathcal{P}_{[N]}^{(1)}}
\newcommand{\pim}{\mathcal{P}_{[i-1]}^{(1)}}
\newcommand{\pip}{\mathcal{P}_{[i]}^{(1)}}
\newcommand{\yex}{y^{\text{extr}}}
\newcommand{\cheb}{\mathcal{C}^{\tau,G}}

\newcommand{\chebpoly}{\mathcal{T}^{\tau,G}}

\newcommand{\cl}{c^{\lambda}}

\begin{document}
\title{Nonlinear Acceleration of Momentum and Primal-Dual Algorithms}

\author{Raghu Bollapragada}
 \address[Raghu Bollapragada:]{Corresponding author. \\ Mathematics and Computer Science Division,
 Argonne National Laboratory, Lemont, IL, USA. \\The author was a PhD student in the department of Industrial Engineering and Management Sciences at Northwestern University, IL, USA, when this work was done.}
 \email[corresponding author]{raghu.bollapragada@u.northwestern.edu}

 \author{Damien Scieur}
 \address[Damien Scieur:]{SAMSUNG SAIL, Montreal, Canada.\\ This author was a PhD student at INRIA \& D.I., UMR 8548,
 \'Ecole Normale Sup\'erieure, Paris, France, when this work was done.}
 \email{damien.scieur@gmail.com}

 \author{Alexandre d'Aspremont}
 \address[Alexandre d'Aspremont:]{CNRS \& D.I., UMR 8548,
 \'Ecole Normale Sup\'erieure, Paris, France.}
 \email{aspremon@gmail.com}


\keywords{}
\date{\today}
\subjclass[2010]{}

\begin{abstract}
We describe convergence acceleration schemes for multistep optimization algorithms. The extrapolated solution is written as a nonlinear average of the iterates produced by the original optimization method. Our analysis does not need the underlying fixed-point operator to be symmetric, hence handles e.g. algorithms with momentum terms such as Nesterov's accelerated method, or primal-dual methods. The weights are computed via a simple linear system and we analyze performance in both online and offline modes. We use Crouzeix's conjecture to show that acceleration performance is controlled by the solution of a Chebyshev problem on the numerical range of a non-symmetric operator modeling the behavior of iterates near the optimum. Numerical experiments are detailed on logistic regression problems. 
\end{abstract}
\maketitle

\section{Introduction}\label{s:intro}
Extrapolation techniques, such as Aitken's~$\Delta^2$ or Wynn's $\varepsilon$-algorithm, provide an improved estimate of the limit of a sequence using its last few iterates, and we refer the reader to \citep{brezinski2013extrapolation} for a complete survey. These methods have been extended to vector sequences, where they are known under various names, e.g. Anderson acceleration \citep{anderson1965iterative,walker2011anderson}, minimal polynomial extrapolation \citep{cabay1976polynomial} or reduced rank extrapolation \citep{eddy1979extrapolating}.

Classical optimization algorithms typically retain only the last iterate or the average of iterates \citep{polyak1992acceleration} as their best estimate of the optimum, throwing away all the information contained in the converging sequence. This is highly wasteful from a statistical perspective and extrapolation schemes estimate instead the optimum using a weighted average of the last iterates produced by the underlying algorithm, where the weights depend on the iterates (i.e. a {\em nonlinear} average). Overall, computing those weights means solving a small linear system so nonlinear acceleration has marginal computational complexity.

Recent results by \citep{scieur2016regularized} adapted classical extrapolation techniques related to Aitken's~$\Delta^2$, Anderson's method and minimal polynomial extrapolation to design extrapolation schemes for accelerating the convergence of basic optimization methods such as gradient descent. They showed that using only iterates from fixed-step gradient descent, extrapolation algorithms achieve the optimal convergence rate of \citep{nesterov2013introductory} {\em without any modification to the original algorithm}. However, these results are only applicable to iterates produced by single-step algorithms such as gradient descent, where the underlying operator is symmetric, thus excluding much faster momentum-based methods such as SGD with momentum or Nesterov's algorithm. Our results here seek to extend those of \citep{scieur2016regularized} to multistep methods, i.e. to accelerate accelerated methods.

Our contribution here is twofold. First, we show that the accelerated convergence bounds in \citep{scieur2016regularized} can be directly extended to multistep methods when the operator describing convergence near the optimum has a particular block structure, by modifying the extrapolating sequence. This result applies in particular to Nesterov's method and the stochastic gradient algorithms with a momentum term. Second, we use Crouzeix's recent results \citep{Crou07,Crou17,Gree17} to show that, in the general non-symmetric case, acceleration performance is controlled by the solution of a Chebyshev problem on the numerical range of the linear, non-symmetric operator modelling the behavior of iterates near the optimum. We characterize the shape of this numerical range for various classical multistep algorithms such as Nesterov's method \citep{Nest83}, and Chambolle-Pock's algorithm \citep{chambolle2011first}.

We then study the performance of our technique on a  logistic regression problem. The online version (which modifies iterations) is competitive with L-BFGS in our experiments and significantly faster than classical accelerated algorithms. Furthermore, it is robust to miss-specified strong convexity parameters.

\subsection*{Organization of the paper}

In Section~\ref{s:nacc}, we describe the iteration schemes that we seek to accelerate, introduce the Regularized Nonlinear Acceleration (RNA) scheme, and show how to control its convergence rate for linear iterations (e.g. solving quadratic problems). 

In Section~\ref{s:crouzeix} we show how to bound the convergence rate of acceleration schemes on generic nonsymmetric iterates using Crouzeix's conjecture and bounds on the minimum of a Chebyshev problem written on the numerical range of the nonsymmetric operator. We apply these results to Nesterov's method and the Chambolle-Pock primal-dual algorithm in Section~\ref{s:algos}.

We extend our results to generic nonlinear updates using a constrained formulation of RNA (called CNA) in Section~\ref{s:nonlin}. We show optimal convergence rates in the symmetric case for CNA on simple gradient descent with linear combination of previous iterates in Section~\ref{s:grad}, producing a much cleaner proof of the results in~\citep{scieur2016regularized} on RNA. In Section~\ref{s:online}, we show that RNA can be applied online, i.e. that we can extrapolate iterates produced by an extrapolation scheme at each iteration (previous results only worked in batch mode) and apply this result to speed up Nesterov's method.

\section{Nonlinear Acceleration}\label{s:nacc}
We begin by describing the iteration template for the algorithms to which we will apply acceleration schemes. 

\subsection{General setting} Consider the following optimization problem
\BEQ\label{eq:fprob}
\min_{x\in\reals^n} f(x)
\EEQ
in the variable $x\in\reals^n$, where $f(x)$ is strongly convex with parameter~$\mu$ with respect to the Euclidean norm, and has a Lipschitz continuous gradient with parameter $L$ with respect to the same norm. 
We consider the following class of algorithms, written
\BEQ \label{eq:general_iteration}
\left\{\BA{l}
    x_{i} = g(y_{i-1})\\
    y_i = \textstyle \sum_{j=1}^i \alpha_j^{(i)} x_j + \beta_j^{(i)} y_{j-1},
\EA\right.
\EEQ
where $x_i,y_i\in\reals^d$ and $g: \reals^d \to \reals^d$ is an iterative update, potentially stochastic. For example, $g(x)$ can be a gradient step with fixed stepsize, in which case $g(x)=x - h\nabla f(x)$. We assume the following condition on the coefficients $\alpha$ and $\beta$, to ensure consistency \citep{scieur2017integration},
\[
    \textbf{1}^T(\alpha + \beta) = 1, \quad \forall k, \ \alpha_j \neq 0.
\]
We can write these updates in matrix format, with
\BEQ\label{eq:def_xy}
    X_i = [x_1,x_2,\ldots, x_i], \quad  Y_i = [y_0,y_1,\ldots, y_{i-1}]. 
\EEQ
Using this notation, \eqref{eq:general_iteration} reads (assuming $x_0=y_0$)
\BEA\label{eq:general_iteration_matrix}
    X_i = g(Y_{i-1})\,, \qquad Y_i = [x_0, X_{i}]L_i,
\EEA
where $g(Y)$ stands for $[g(y_0),g(y_1),\ldots,g(y_{i-1})]$ and the matrix $L_i$ is upper-triangular of size $i\times i$ with nonzero diagonal coefficients, with columns summing to one. The matrix $L_i$ can be constructed iteratively, following the recurrence
\BEQ
    L_i = \begin{bmatrix} L_{i-1} & \alpha_{[1:i-1]} + L_{i-1}\beta \\
    0_{1\times i-1} & \alpha_i
    \end{bmatrix}, \quad L_0 = 1.\label{eq:recurence_L}
\EEQ
In short, $L_i$ gathers coefficients from the linear combination in~\eqref{eq:general_iteration}. This matrix, together with $g$, characterizes the algorithm.

The iterate update form \eqref{eq:general_iteration} is generic and includes many classical algorithms such as the accelerated gradient method in \citep{nesterov2013introductory}, where
\[
\begin{cases}
    x_{i} &= g(y_{i-1}) = y_{i-1} - \frac{1}{L} \nabla f(y_{i-1}) \\
    y_{i} &= \left(1+\frac{i-1}{i+2}\right)x_{i} - \frac{i-1}{i+2}\; x_{i-1}.
\end{cases}
\]
As in~\citep{scieur2016regularized} we will focus on improving our estimates of the solution to problem~\eqref{eq:fprob} by tracking only the sequence of iterates $(x_i, y_i)$ produced by an optimization algorithm, without any further oracle calls to $g(x)$. The main difference with the work of \citep{scieur2016regularized} is the presence of a linear combination of previous iterates in the definition of $y$ in \eqref{eq:general_iteration}, so the mapping from $x_{i}$ to $x_{i+1}$ is usually \textit{non-symmetric}. For instance, for Nesterov's algorithm, the Jacobian of $x_{i+1}$ with respect to $x_i$, $y_i$ reads
\[
    J_{x_{i+1}} = \begin{bmatrix}
        0 & J_g\\
        \left(1+\frac{i-2}{i+1}\right) \textbf{I} & - \frac{i-2}{i+1}\textbf{I}
    \end{bmatrix} \neq J_{x_{i+1}}^T
\]
where $J_{x_{i+1}}$ is the Jacobian of the function $g$ evaluated at $x_{i+1}$. In what follows, we show that looking at the residuals
\BEA
    r(x) \triangleq g(x)-x, \quad r_i = r(y_{i-1}) = x_i-y_{i-1}, \quad R_i = [r_1\ldots r_i], \label{eq:residue}
\EEA
allows us to recover the convergence results from \citep{scieur2016regularized} when the Jacobian of the function $g$, written $J_g$, is symmetric. Moreover, we extend the analysis for \textit{non symmetric} Jacobians. This allows us to accelerate for instance accelerated methods or primal-dual methods. We now briefly recall the key ideas driving nonlinear acceleration schemes.
\input{nonlinear_acceleration.tex}

\input{Crouzeix_Conjecture.tex}

\input{perturbation_analysis.tex}

\input{numexp_short.tex}

\section*{Acknowledgements}
The authors are very grateful to Lorenzo Stella for fruitful discussions on acceleration and the Chambolle-Pock method. AA is at CNRS \& d\'epartement d'informatique, \'Ecole normale sup\'erieure, UMR CNRS 8548, 45 rue d'Ulm 75005 Paris, France,  INRIA and PSL Research University. The authors would like to acknowledge support from the {\em ML \& Optimisation} joint research initiative with the {\em fonds AXA pour la recherche} and Kamet Ventures, as well as a Google focused award. DS was supported by a European Union Seventh Framework Programme (FP7- PEOPLE-2013-ITN) under grant agreement n.607290 SpaRTaN. RB was a PhD student at Northwestern University at the time this work was completed and was supported by Department of Energy grant DE-FG02-87ER25047 and DARPA grant 650-4736000-60049398. 

{\small \bibliographystyle{plainnat}
\bibsep 1ex
\bibliography{biblio,MainPerso,bib2,bib3}}

\end{document}

%% file: defs.tex
\definecolor{ddarkbrown}{rgb}{0.5,0.2,0.05} \definecolor{bbluegray}{rgb}{0.05,0,0.5}

\newtheorem{theorem}{Theorem}[section]
\newtheorem{proposition}[theorem]{Proposition}
\newtheorem{definition}[theorem]{Definition}

\renewenvironment{proof}{\textbf{Proof.}}{\QED\bigskip}

\newcommand{\BEAS}{\begin{eqnarray*}}
\newcommand{\EEAS}{\end{eqnarray*}}
\newcommand{\BEA}{\begin{eqnarray}}
\newcommand{\EEA}{\end{eqnarray}}
\newcommand{\BEQ}{\begin{equation}}
\newcommand{\EEQ}{\end{equation}}
\newcommand{\BIT}{\begin{itemize}}
\newcommand{\EIT}{\end{itemize}}
\newcommand{\BNUM}{\begin{enumerate}}
\newcommand{\ENUM}{\end{enumerate}}

\newcommand{\BA}{\begin{array}}
\newcommand{\EA}{\end{array}}

\newcommand{\reals}{{\mathbb R}}
\newcommand{\complexs}{{\mathbb C}}

\newcommand{\Tr}{\mathop{\bf Tr}}

\newcommand{\Co}{{\mathop {\bf Co}}}

\newcommand{\QED}{~~\rule[-1pt]{6pt}{6pt}}

\newcommand{\argmin}{\mathop{\rm argmin}}

\newcommand{\vect}{\mathop{\bf vec}}



%% file: nonlinear_acceleration.tex
\subsection{Linear Algorithms}
In this section, we focus on iterative algorithms $g$ that are linear, i.e., where
\BEQ
    g(x) = G(x-x^*) + x^*. \label{eq:linear_g}
\EEQ
The matrix $G$ is of size $d\times d$, and, contrary to \citep{scieur2016regularized}, we do not assume symmetry. Here, $x^*$ is a fixed point of $g$. In optimization problems where $g$ is typically a gradient mapping $x^*$ is the minimum of an objective function.  Its worth mentioning that \eqref{eq:linear_g} is equivalent to $Ax+b$, thus we do not require $x^*$ to evaluate the mapping $g(x)$. We first treat the case where $g(x)$ is linear, as the nonlinear case will then be handled as a perturbation of the linear one. 

We introduce $\pnp$, the set of all polynomials $p$ whose degree is \textit{exactly} $N$ (i.e., the leading coefficient is nonzero), and whose coefficients sum to one. More formally,
\BEQ
    \pnp = \{ p \in \reals[x]: \deg(p) = N,\,  p(1) = 1  \}.
\EEQ
The following proposition extends a result by \citet{scieur2016regularized} showing that iterates in~\eqref{eq:general_iteration} can be written using polynomials in $\pnp$. This formulation is helpful to derive the rate of converge of the Nonlinear Acceleration algorithm.
\begin{proposition} \label{prop:poly_iter}
Let $g$ be the linear function \eqref{eq:linear_g}. Then, the $N$-th iteration of \eqref{eq:general_iteration} is equivalent to
\BEA \label{eq:polynomial_iteration}
    x_N = x^* + G(y_{N-1}-x^*), \qquad y_N = x^* + p_N(G)(x_0-x^*),\quad \mbox{for some $p_N\in \pnp$.}
\EEA
\end{proposition}
\begin{proof}
    We prove \eqref{eq:polynomial_iteration} iteratively. Of course, at iteration zero,
    \[
        y_0 = x^* + 1 \cdot (x_0 - x^*),
    \]
    and $1$ is indeed  polynomial of degree zero whose coefficient sum to one. Now, assume
    \[
        y_{i-1} - x^* = p_{i-1}(G)(x_0-x^*), \quad p_{i-1} \in \pim.
    \]
    We show that
    \[
        y_{i} - x^* = p_{i}(G)(x_0-x^*), \quad p_{i} \in \pip.
    \]
    By definition of $y_{i}$ in~\eqref{eq:general_iteration},
    \[
        y_{i} - x^* = \textstyle \sum_{j=1}^i \alpha_j^{(i)} x_j + \beta_j^{(i)} y_{i-1} - x^*,
    \]
    where $(\alpha + \beta)^T\textbf{1} = 1$. This also means that
    \[
        y_{i} - x^* = \textstyle \sum_{j=1}^i \alpha_j^{(i)} (x_j-x^*) + \beta_j^{(i)} (y_{j-1}-x^*) .
    \]
    By definition, $x_{j} -x^* = G(y_{j-1} - x^*) $, so
    \[
        y_{i} - x^* = \textstyle \sum_{j=1}^i \left(\alpha_j^{(i)} G  + \beta_j^{(i)} I\right) (y_{j-1}-x^*) .
    \]
    By the recurrence assumption,
    \[
        y_{i} - x^* = \textstyle \sum_{j=1}^i \left(\alpha_j^{(i)} G  + \beta_j^{(i)} I\right) p_{j-1}(G) (x_0-x^*),
    \]
    which is a linear combination of polynomials, thus $y_{i} - x^* = p(G)(x_0-x^*)$. It remains to show that $p \in \pip$. Indeed,
    \[
        \deg (p) = \max_j   \max \left\{(1+\deg (p_{j-1}(G))) 1_{\alpha_j \neq 0},\;\; \deg (p_{j-1}(G)) 1_{\beta_j \neq 0}\right\},
    \]
    where $1_{\alpha_j \neq 0} = 1$ if $\alpha_j \neq 0$ and $0$ otherwise. By assumption, $\alpha_i \neq 0$ thus
    \[
        \deg (p) \geq 1+\deg (p_{i-1}(G)) = i.
    \]
    Since $p$ is a  linear combination of polynomials of degree at most $i$, \[
        \deg (p) = i.
    \]
    It remains to show that $p(1)=1$. Indeed,
    \[
        p(1) = \sum_{j=1}^i \left(\alpha_j^{(i)} 1  + \beta_j^{(i)} \right) p_{j-1}(1).
    \]
    Since $\left(\alpha_j^{(i)} 1  + \beta_j^{(i)} \right) = 1$ and $p_{j-1}(1) = 1$, $p(1) = 1$ and this proves the proposition.
\end{proof}

\subsection{Regularized Nonlinear Acceleration Scheme}
We now propose a modification of RNA that can accelerate any algorithm of the form \eqref{eq:general_iteration} by combining the approaches of \cite{anderson1965iterative} and \cite{scieur2016regularized}. We introduce a mixing parameter~$\eta$, as in Anderson acceleration (which only impact the constant term in the rate of convergence). Throughout this paper, \textbf{RNA} will refer to Algorithm \ref{algo:rna} below. 

\begin{algorithm}[htb]
\caption{Regularized Nonlinear Acceleration (\textbf{RNA})}
\label{algo:rna}
\begin{algorithmic}[1]
   \STATE {\bfseries Data:} Matrices $X$ and $Y$ of size $d\times N$ constructed from the iterates as in~\eqref{eq:general_iteration} and~\eqref{eq:def_xy}.
   \STATE {\bfseries Parameters:} Mixing $\eta\neq 0$, regularization $\lambda \geq 0$.\\
   \hrulefill
   \STATE \textbf{1.} Compute matrix of residuals $R = X-Y$.
   \STATE \textbf{2.} Solve
   \BEQ
        \cl = \frac{(R^TR+(\lambda\|R\|^2_2) I)^{-1} \textbf{1}_N}{\textbf{1}_N^T(R^TR+(\lambda\|R\|^2_2) I)^{-1}\textbf{1}_N}. \label{eq:cl}
    \EEQ
    \STATE \textbf{3.} Compute extrapolated solution $\yex = (Y-\eta R)\cl$.
\end{algorithmic}
\end{algorithm}

\subsection{Computational Complexity} 

\citet{scieur2016regularized} discuss the complexity of Algorithm \ref{algo:rna} in the case where $N$ is small (compared to $d$). When the algorithm is used once on $X$ and $Y$ (batch acceleration), the computational complexity is $O(N^2 d)$, because we have to multiply $R^T$ and $R$. However, when Algorithm \eqref{algo:rna} accelerates iterates on-the-fly, the matrix $R^TR$ can be updated using only $O(Nd)$ operations. The complexity of solving the linear system is negligible as it takes only $O(N^3)$ operation. Even if the cubic dependence is bad for large $N$, in our experiment $N$ is typically equal to 10, thus adding a negligible computational overhead compared to the computation of a gradient in large dimension which is higher by orders.

\subsection{Convergence Rate}
We now analyze the convergence rate of Algorithm \ref{algo:rna} with $\lambda = 0$, which corresponds to Anderson acceleration \citep{anderson1965iterative}. In particular, we show its optimal rate of convergence when $g$ is a linear function. In the context of optimization, this is equivalent to the application of gradient descent for minimizing quadratics. Using this special structure, the iterations~\eqref{eq:polynomial_iteration} produce a sequence of polynomials and the next theorem uses this special property to bound the convergence rate. Compared to previous work in this vein \citep{scieur2016regularized,scieur2017nonlinear} where the results only apply to algorithm of the form $x_{i+1} = g(x_i)$, this theorem applies to \textit{any} algorithm of the class \eqref{eq:general_iteration} where in particular, we allow $G$ to be nonsymmetric.

\begin{theorem} \label{thm:optimal_rate}
    Let $X$, $Y$ in~\eqref{eq:def_xy} be formed using iterates from \eqref{eq:general_iteration}. Let $g$ be defined in~\eqref{eq:linear_g}, where $G\in \mathbb{R}^{d\times d}$ does not have $1$ as eigenvalue. The norm of the residual of the extrapolated solution $\yex$, written
    \[
        r(\yex) = g(\yex) - \yex,
    \]
    produced by Algorithm \ref{algo:rna} with $\lambda = 0$,  is bounded by
    \[
        \|r(\yex)\|_2 \leq \| I-\eta (G-I)  \|_2 ~\| p^*_{N-1}(G)r(x_0) \|_2,
    \]
    where $p^*_{N-1}$ solves
    \BEQ
        \textstyle p^*_{N-1} = \argmin_{p\in \pnm} \| p(G)r(x_0)  \|_2. \label{eq:minimal_polynomial}
    \EEQ
    Moreover, after at most $d$ iterations, the algorithm converges to the exact solution, satisfying $\|r(\yex)\|_2 = 0$. 
\end{theorem}
\begin{proof}
    First, we write the definition of $\yex$ from Algorithm \ref{algo:rna} when $\lambda = 0$,
    \[
        \yex -x^* = (Y-\eta R) c - x^*.
    \]
    Since $c^T\textbf{1} = 1$, we have $X^*c = x^*$, where $X^* = [x^*,x^*,\ldots, x^*,]$. Thus,
    \[
        \yex -x^* = (Y-X^*-\eta R) c.
    \]
    Since $R = G(Y-X^*)$,
    \[
        \yex - x^* = (I-\eta (G-I))(Y-X^*) c .
    \]
    We have seen that the columns of $Y-X^*$ are polynomials of different degrees, whose coefficients sums to one \eqref{eq:polynomial_iteration}. This means
    \[
        \yex - x^* = (I-\eta (G-I))\sum_{i=0}^{N-1} c_i p_i(G)(x_0-x^*). 
    \]
    In addition, its residual reads
    \BEAS
        r(\yex) & = & (G-I)(\yex - x^*) \\
        & = & (G-I)(I-\eta (G-I))\sum_{i=0}^{N-1} p_i(G)(x_0-x^*) \\
        & = & (I-\eta (G-I)\sum_{i=0}^{N-1} c_i p_i(G)r(x_0).
    \EEAS
    Its norm is thus bounded by
    \[
        \|r(\yex)\| \leq \|I-\eta (G-I)\| \|\underbrace{\sum_{i=0}^{N-1} c_i p_i(G)r(x_0)}_{=Rc}\|.
    \]
    By definition of $c$ from Algorithm \ref{algo:rna},
    \[
        \|r(\yex)\| \leq \|I-\eta (G-I)\| \min_{c:c^T\textbf{1} = 1} \|\sum_{i=0}^{N-1} c_i p_i(G)r(x_0)\|.
    \]
    Because $p_i$ are all of degree exactly equal to $i$, the $p_i$ are a basis of the set of all polynomial of degree at most $N-1$. In addition, because $p_i(1) = 1$, restricting the sum of coefficients $c_i$ to $1$ generates the set $\pnm$. We have thus
    \[
        \|r(\yex)\| \leq \|I-\eta (G-I)\| \min_{p\in \pnm } \|p(G)r_0\|.
    \]
    Finally, when $N>d$, it suffice to take the minimal polynomial of the matrix $G$ named $p_{\min,G}$, whose coefficient are normalized by $p_{\min,G}(1)$. Since the eigenvalues of $G$ are strictly inferior to $1$, $p_{\min,G}(1)$ cannot be zero. 
\end{proof}

In optimization, the quantity $\|r(\yex)\|_2$ is proportional to the norm of the gradient of the objective function computed at $\yex$. This last theorem reduces the analysis of the rate of convergence of RNA to the analysis of the quantity \eqref{eq:minimal_polynomial}. In the symmetric case discussed in~\citep{scieur2016regularized}, this bound recovers the optimal rate in \citep{nesterov2013introductory} which also appears in the complexity analysis of Krylov methods (like GMRES or conjugate gradients \citep{golub1961chebyshev,golub2012matrix}) for quadratic minimization.

%% file: Crouzeix_Conjecture.tex
\section{Crouzeix's Conjecture \& Chebyshev Polynomials on the Numerical Range}\label{s:crouzeix}
We have seen in~\eqref{eq:minimal_polynomial} from Theorem~\ref{thm:optimal_rate} that the convergence rate of nonlinear acceleration is controlled by the norm of a matrix polynomial in the operator $G$, with  
\[
    \|r(\yex)\|_2 \leq \| I-\eta (G-I)  \|_2 ~\| p^*_{N-1}(G)r(x_0) \|_2,
\]
where $r(\yex)=\yex - g(\yex)$ and $p^*_{N-1}$ solves
\[
    p^*_{N-1} = \argmin_{p\in \pnm} \| p(G)r(x_0)  \|_2. 
\]
The results in~\citep{scieur2016regularized} recalled above handle the case where the operator $G$ is {\em symmetric}. Bounding $\|p(G)\|_2$ when $G$ is non-symmetric is much more difficult. Fortunately, Crouzeix's conjecture \citep{Crou04} allows us to bound $\|p(G)\|_2$ by solving a Chebyshev problem on the numerical range of $G$, in the complex plane. 

\begin{theorem}[\citet{Crou04}]
Let $G\in\complexs^{n\times n}$, and $p(x)\in \complexs[x]$, we have
\[
\|p(G)\|_2 \leq c\max_{z \in W(G)} |p(z)|
\]
for some absolute constant $c\geq 2$.
\end{theorem}\label{th:crouzeix}
Here $W(G)\subset \complexs$ is the numerical range of the matrix $G\in\reals^{n\times n}$, i.e. the range of the Rayleigh quotient
\BEQ\label{eq:numrange}
W(G) \triangleq \left\{ x^*Gx: \|x\|_2=1, x \in \complexs^n \right\}.
\EEQ
\citep{Crou07} shows $c\leq 11.08$ and Crouzeix's conjecture states that this can be further improved to $c=2$, which is tight. A more recent bound in \citep{Crou17} yields $c=1+\sqrt{2}$ and there is significant numerical evidence in support of the $c=2$ conjecture \citep{Gree17}. This conjecture has played a vital role in providing convergence results for e.g. the GMRES method \citep{saad1986gmres,choi2015roots}. 

Crouzeix's result allows us to turn the problem of finding uniform bounds for the norm of the matrix polynomial $\|p(G)\|_2$ to that of bounding $p(z)$ over the numerical range of $G$ in the complex plane, an arguably much simpler two-dimensional Chebyshev problem.

\subsection{Numerical Range Approximations}
The previous result links the convergence rate of accelerated algorithms with the optimum value of a Chebyshev problem over the numerical range of the operator $G$ and we now recall classical methods for computing the numerical range. There are no generic tractable methods for computing the exact numerical range of an operator $G$. However, efficient numerical methods approximate the numerical range based on key structural properties. The Toeplitz-Hausdorff theorem \citep{hausdorff1919wertvorrat, toeplitz1918algebraische} in particular states that the numerical range $W(G)$ is a closed convex bounded set. Therefore, it suffices to characterize points on the boundary, the convex hull then yields the numerical range. 

\citet{johnson1978numerical} made the following observations using the properties of the numerical range,
\begin{align}
    \max_{z \in W(G)} Re(z) &= \max_{r \in W(H(G))} r= \lambda_{max}(H(G)) \label{eq: maxrealval}\\
    W(e^{i\theta}G) &= e^{i\theta} W(G),\quad  \forall \theta \in [0, 2\pi), \quad \label{eq: fieldvaluesrotation}
\end{align}
where $Re(z)$ is the real part of complex number $z$, $H(G)$ is the Hermitian part of $G$, i.e. $ H(G) = ({G + G^*})/{2}$
and $\lambda_{max}(H(G))$ is the maximum eigenvalue of $H(G)$. The first property implies that the line parallel to the imaginary axis is tangent to $W(G)$ at $\lambda_{max}(H(G))$. The second property can be used to determine other tangents via rotations. Using these observations \cite{johnson1978numerical} showed that the points on the boundary of the numerical range can be characterized as 
$
    p_\theta =\{v_\theta^*Gv_\theta : \theta \in [0, 2\pi)\}
$
 where $v_\theta$ is the normalized eigenvector corresponding to the largest eigenvalue of the Hermitian matrix
\begin{equation}
    H_\theta = \frac{1}{2}(e^{i\theta}G + e^{-i\theta}G^*)
\end{equation}
The numerical range can thus be characterized as follows.

\begin{theorem} \citep{johnson1978numerical} For any $G\in\complexs^{n\times n}$, we have
\[
 W(G) = Co\{p_\theta  : 0\leq \theta < 2\pi\}
\]
where $Co\{Z\}$ is the convex hull of the set $Z$. 
\end{theorem}
Note that $p_\theta$ cannot be uniquely determined as the eigenvectors $v_\theta$ may not be unique but the convex hull above is uniquely determined. 

\subsection{Chebyshev Bounds \& Convergence Rate}
Crouzeix's result means that bounding the convergence rate of accelerated algorithms can be achieved by bounding the optimum of the Chebyshev problem 
\BEQ\label{eq:cheb-C}
\min_{\substack{p \in\complexs[z]\\p(1)=1}} ~~ \max_{z\in W(G)} |p(z)|
\EEQ
where $G\in\complexs^{n \times n}$. This problem has a trivial answer when the numerical range $W(G)$ is spherical, but the convergence rate can be significantly improved when $W(G)$ is less isotropic. 

\subsubsection{Exact Bounds on Ellipsoids}
We can use an outer ellipsoidal approximation of $W(G)$, bounding the optimum value of the Chebyshev problem~\eqref{eq:cheb-C} by
\BEQ\label{eq:ChebE}
\min_{\substack{p(z)\in\complexs[x]\\p(1)=1}} ~~ \max_{z\in\mathcal{E}_r} |p(z)|
\EEQ
where
\BEQ\label{eq:Er}
\mathcal{E}_r\triangleq \{z\in\complexs:|z-1|+|z+1| \leq r+ 1/r\}.
\EEQ
This Chebyshev problem has an explicit solution in certain regimes. As in the real case, we will use $C_n(z)$, the Chebyshev polynomial of degree $k$.
\citet{Fisc91} show the following result on the optimal solution to problem ~\eqref{eq:ChebE} on ellipsoids. 

\begin{theorem}\citep[Th.\,2]{Fisc91}\label{th:fish}
Let $k\geq 5$, $r>1$ and $c \in \reals$. The polynomial 
\[
    T_{k,\kappa}(z)=T_k(z)/T_k(1-\kappa) 
\]
where 
\[
T_k(z)= \frac{1}{2}\left(v^k + \frac{1}{v^k}\right), \quad v =\frac{1}{2}\left( z + \frac{1}{z}\right)
\]
is the unique solution of problem~\eqref{eq:ChebE} if either 
\[
|1-\kappa| \geq \frac{1}{2}\left(r^{\sqrt{2}} + r^{-\sqrt{2}}\right) 
\]
or
\[
|1-\kappa| \geq \frac{1}{2 a_r}\left(2a_r^2 - 1 + \sqrt{2a_r^4-a_r^2+1}\right)
\]
where $a_r=(r+1/r)/2.$
\end{theorem}

The optimal polynomial for a general ellipse $\mathcal{E}$ can be obtained by a simple change of variables. That is, the polynomial $\bar{T}_k(z)={T_k(\frac{c-z}{d})}/{T_k(\frac{c-1}{d})}$ is optimal for the problem \eqref{eq:ChebE} over any ellipse $\mathcal{E}$ with center $c$, focal distance $d$ and semi-major axis $a$. It can be easily seen that the maximum value is achieved at the point $a$ on the real axis. That is the solution to the min max problem is given by $\bar{T}_k(a)$. Figure \ref{fig:chebyshev_cont5} shows the surface of the optimal polynomial with degree~$5$ for $a=0.8, d=0.76$ and $c=0$.
\begin{figure}[h!t]
    \centering
    \includegraphics[width=0.47\textwidth]{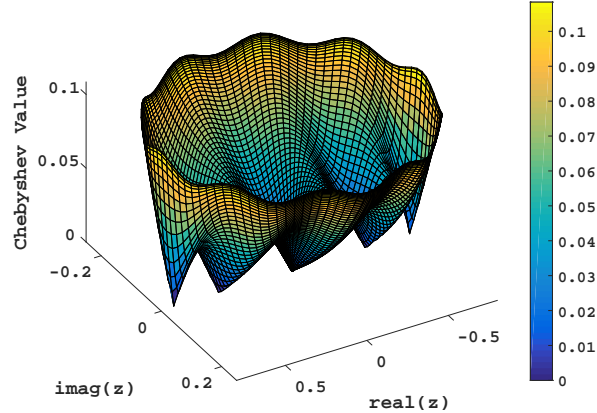}
    \caption{Surface of the optimal polynomial $\bar{T}_n(z)$ with degree $5$ for $a=0.8, d=0.76$ and $c=0$. }
    \label{fig:chebyshev_cont5}
\end{figure}

Figure \ref{fig:chebyshev_ecc5} shows the solutions to the problem \eqref{eq:ChebE} with degree $5$ for various ellipses with center at origin, various eccentricity values $e = d/a$ and semi-major axis $a$.
\begin{figure}[h!t]
    \centering
    \includegraphics[width=0.47\textwidth]{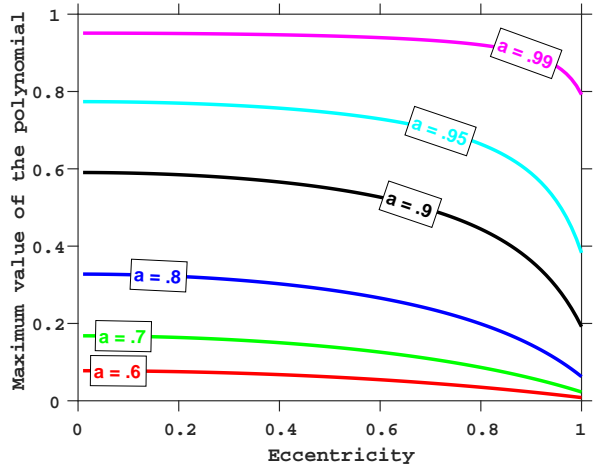}
    \caption{Optimal value of the Chebyshev problem \eqref{eq:ChebE} for ellipses with centers at origin. Lower values of the maximum of the Chebyshev problem mean faster convergence. The higher the eccentricity here, the faster the convergence.}
    \label{fig:chebyshev_ecc5}
\end{figure}
Here, zero eccentricity corresponds to a sphere, while an eccentricity of one corresponds to a line.




\section{Accelerating Non-symmetric Algorithms}\label{s:algos}
We have seen in the previous section that (asymptotically) controlling the convergence rate of the nonlinear acceleration scheme in Algorithm~\ref{algo:rna} for generic operators $G$ means bounding the optimal value of the Chebyshev optimization problem in~\eqref{eq:cheb-C} over the numerical range of the operator driving iterations near the optimum. In what follows, we explicitly detail this operator and approximate its numerical range for two classical algorithms, Nesterov's accelerated method \citep{Nest83} and Chambolle-Pock's Primal-Dual Algorithm \citep{chambolle2011first}. We focus on quadratic optimization below. We will see later in Section~\ref{s:nonlin} that, asymptotically at least, the behavior of acceleration on generic problems can be analyzed as a perturbation of the quadratic case.

\subsection{Nesterov's Accelerated Gradient Method}
The iterates formed by Nesterov's accelerated gradient descent method for minimizing smooth strongly convex functions with constant stepsize follow
\begin{equation}
\left\{
\begin{aligned}
x_k &= y_{k-1} - \alpha \nabla f(y_{k-1}) \\
y_{k} &= x_k + \beta(x_k - x_{k-1})
\end{aligned}
\right.
\label{eq:nest_iterate}
\end{equation}
with $\beta = \frac{\sqrt{L} - \sqrt{\mu}}{\sqrt{L} + \sqrt{\mu}}$, 
where $L$ is the gradient's Lipschitz continuity constant and $\mu$ is the strong convexity parameter. This algorithm is better handled using the results in previous sections, and we only use it here to better illustrate our results on non-symmetric operators.

\subsubsection{Nesterov's Operator in the quadratic case}
When minimizing quadratic functions $f(x) = \frac{1}{2}\|Bx - b\|^2$, using constant stepsize~$1/L$, these iterations become,
\[
\begin{cases}
    x_k - x^* &= y_{k-1} - x^* - \frac{1}{L}B^T(By_{k-1} - b) \\
    y_k - x^* &= x_k - x^* + \beta(x_k - x^* - x_{k-1} + x^*).
\end{cases}
\]
or again,
\BEAS
\begin{bmatrix}
x_{k} - x^*\\
y_{k} - x^*
\end{bmatrix} =
\begin{bmatrix}
0 & A\\
-\beta I & (1 + \beta) A
\end{bmatrix} \begin{bmatrix}
x_{k-1} - x^*\\
y_{k-1} - x^*
\end{bmatrix} 
\EEAS
where $A = I - \frac{1}{L}B^TB$. We write $G$ the {\em non-symmetric} linear operator in these iterations, i.e.
\begin{align}
G = \begin{bmatrix}
0 & A\\
-\beta I & (1 + \beta) A
\end{bmatrix} 
\end{align}
The results in Section~\ref{s:nacc} show that we can accelerate the sequence $z_k = (x_{k},y_{k})$ if the solution to the minmax problem \eqref{eq:cheb-C} defined over the numerical range of the operator $G$ is bounded. 

\subsubsection{Numerical Range}
We can compute the numerical range of the operator $G$ using the techniques described in Section \eqref{s:nacc}. In the particular case of Nesterov's accelerated gradient method, the numerical range is a convex hull of ellipsoids. We show this by considering the $2\times 2$ operators obtained by replacing the symmetric positive matrix $G$ with its eigenvalues, to form
\begin{align} \label{eq:nesteigen}
G_j = \begin{bmatrix}
0 & \lambda_j\\
-\beta I & (1 + \beta) \lambda_j
\end{bmatrix} \quad \text{for } j \in \{1,2,\cdots,n\}
\end{align}
where $0<\lambda_1 \leq \lambda_2 \leq \cdots \leq \lambda_n < 1$ are the eigenvalues of the matrix $A$. We have the following result.

\begin{theorem}
    The numerical range of operator $G$ is given as the convex hull of the numerical ranges of the operators $G_j$, i.e. $W(G) = \Co\{W(G_1),W(G_2),\cdots,W(G_n)\}$.
\end{theorem}
\begin{proof}
Let $v_1,v_2,\cdots,v_n$ be eigen vectors associated with eigen values $\lambda_1,\lambda_2,\cdots,\lambda_n$ of the matrix $A$. We can write 
\begin{equation*}
    A = \sum_{j=0}^{n} \lambda_jv_jv_j^T \qquad I = \sum_{j=0}^{n}v_jv_j^T
\end{equation*}
Let $t \in W(G) \subset \complexs$. By definition of the numerical range, there exists $z \in \complexs^{2n}$ with $z^*z = 1$ and 
\begin{align*}
    t &= z^*\begin{bmatrix}
0 & A\\
-\beta I & (1 + \beta) A
\end{bmatrix}z \\
&= z^*\begin{bmatrix}
0 & \sum_{j=1}^{n}\lambda_j v_j v_j^T\\
-\beta \sum_{j=1}^{n}v_j v_j^T & (1 + \beta) \sum_{j=1}^{n}\lambda_j v_j v_j^T
\end{bmatrix} z \\
&= \sum_{j=0}^{n}z^*\left(\begin{bmatrix}
0 & \lambda_j \\
-\beta  & (1 + \beta) \lambda_j 
\end{bmatrix} \otimes v_jv_j^T\right)\vect([z_1,z_2])\\
&= \sum_{j=0}^{n}z^*\vect\left(
v_jv_j^T [z_1,z_2]\begin{bmatrix}
0 & \lambda_j \\
-\beta  & (1 + \beta) \lambda_j 
\end{bmatrix}^T \right)\\
\end{align*}
and since $v_jv_j^Tv_jv_j^T=v_jv_j^T$, this last term can be written
\begin{align*}
    t &= \sum_{j=0}^{n} \Tr\left(
v_jv_j^T [z_1,z_2]\begin{bmatrix}
0 & \lambda_j \\
-\beta  & (1 + \beta) \lambda_j 
\end{bmatrix}^T [z_1,z_2]^* v_jv_j^T\right)\\
&= \sum_{j=0}^{n} \Tr(v_jv_j^T) \left(
 [v_j^Tz_1,v_j^Tz_2]\begin{bmatrix}
0 & \lambda_j \\
-\beta  & (1 + \beta) \lambda_j 
\end{bmatrix}^T [z_1^*v_j,z_2^*v_j]^T \right)\\
\end{align*}
Now, let $w_j=[z_1^*v_j,z_2^*v_j]^T$, and 
\begin{equation*}
    y_j = \frac{w_j^TG_jw_j}{\|w_j\|_2^2}
\end{equation*}
and by the definition of the numerical range, we have $y_j \in W(G_j)$. Therefore,
\begin{align*}
    t &= \sum_{j=0}^{n}\left(\frac{w_j^TG_jw_j}{\|w_j\|_2^2}\right)\|w_j\|_2^2
\end{align*}
hence 
\[
    t \in \Co(W(G_1), W(G_2),\cdots,W(G_n)).
\]
We have shown that if $t \in W(G)$ then $t \in \Co(W(G_1), W(G_2),\cdots,W(G_n))$. 
We can show the converse by following the above steps backwards. That is, if $t \in \Co(W(G_1), W(G_2),\cdots,W(G_n))$ then we have,
\begin{align*}
    t = \sum_{j=0}^{n} \theta_j \left(\frac{w_j^TG_jw_j}{\|w_j\|_2^2}\right)
\end{align*}
where $\theta_j > 0$, $\sum_{j=0}^{n}\theta_j =1$ and $w_j \in \complexs^{2}$. Now, let 
\begin{align*}
    z = \sum_{j=0}^{n}\frac{\vect(v_jw_j^T)\theta_j^{1/2}}{\|w_j\|}
\end{align*}
and we have,
\begin{align*}
    t = \sum_{j=0}^{n}[z_1^*v_j z_2^*v_j]G_j\begin{bmatrix}
    v_j^T z_1 \\
    v_j^T z_2
    \end{bmatrix} 
\end{align*}
wherein we used the fact that $v_j^Tv_k = 0$ for any $j \neq k$ and $v_j^Tv_j =1$ in computing $w_j^T = [z_1^*v_j z_2^*v_j]$. We also note that $z^*z =1$ by the definition of $z$ and rewriting the sum in the matrix form we can show that $t \in W(G)$ which completes the proof.  
\end{proof}

To minimize the solution of the Chebyshev problem in~\eqref{eq:cheb-C} and control convergence given the normalization constraint $p(1)=1$, the point $(1,0)$ should be outside the numerical range. Because the numerical range is convex and symmetric w.r.t. the real axis (the operator $G$ is real), this means checking if the maximum real value of the numerical range is less than $1$. 

For $2\times 2$ matrices, the boundary of the numerical range is given by an ellipse \citep{donoghue1957}, so the numerical range of Nesterov's accelerated gradient method is the convex hull of ellipsoids. The ellipse in \citep{donoghue1957} can be determined directly from the entries of the matrix as in \cite{johnson1974computation},as follows.

\begin{theorem}\citep{johnson1974computation}\label{th:jhon}
For any real 2 by 2 matrix 
\[
\begin{bmatrix}
a & b\\
c & d
\end{bmatrix}
\]
the boundary of the numerical range is an ellipse whose axes are the line segments joining the points x to y and w to z respectively where,
\begin{align*}
    x &= \frac{1}{2}(a + d - ((a -d)^2 + (b+c)^2)^{1/2})\\
    w &= \frac{a+d}{2} - i\left|\frac{b-c}{2}\right|\\
    y &= \frac{1}{2}(a + d + ((a -d)^2 + (b+c)^2)^{1/2})\\
    z &= \frac{a+d}{2} + i\left|\frac{b-c}{2}\right|
\end{align*}
are the points in the complex plane. 
\end{theorem}

This allows us to compute the maximum real value of $W(G)$, as the point of intersection of $W(G)$ with the real line which can be computed explicitly as, 
\BEAS
    re(G) &=& \max Re(W(G)) = \max_j \mathop{Re}(W(G_j)) \\
    &=& \frac{1}{2}\left((1 + \beta)\lambda_n + \sqrt{\lambda_n^2(1+\beta)^2 + (\lambda_n - \beta)^2}\right)
\EEAS
where $\lambda_n = 1 - \frac{\mu}{L}$. 

We observe that $re(G)$ is a function of the condition number of the problem and takes values in the interval $[0, 2]$. Therefore, RNA will only work on Nesterov's accelerated gradient method when $re(G) < 1$ holds, which implies that the condition number of the problem $\kappa = \frac{L}{\mu}$ should be less than around $~2.5$ which is highly restrictive. 

An alternative approach is to use RNA on a sequence of iterates sampled every few iterations, which is equivalent to using powers of the operator~$G$. We expect the numerical radius of some power of operator~$G$ to be less than 1 for any conditioning of the problem. This is because the iterates are converging at an $R-$linear rate and so the norm of the power of the operator is decreasing at an $R-$linear rate with the powers. Therefore, using the property that the numerical radius is bounded by the norm of the operator we have,
\begin{equation*}
    re(G^p) = \max Re(W(G^p)) \leq r_{G^p} \leq \|G^p\| \leq C_p \rho^p
\end{equation*}
where $r_{G^p}$ is the numerical radius of $G^p$. Figure \ref{fig:fieldvals_random50_nest} shows the numerical range of the powers of the operator~$G$ for a random matrix $B^TB$ with dimension $d = 50$. We observe that after some threshold value for the power~$p$, $(1,0)$ lies outside the field values corresponding to $G^p$ thus guaranteeing that the acceleration scheme will work. We also observe that the boundaries of the field values are almost circular for higher powers $p$, which is consistent with results on optimal matrices in~\citep{Lewi18}. When the numerical range is circular, the solution of the Chebyshev problem is trivially equal to $z^p$ so RNA simply picks the last iterate and does not accelerate convergence.

\begin{figure}[h!t]
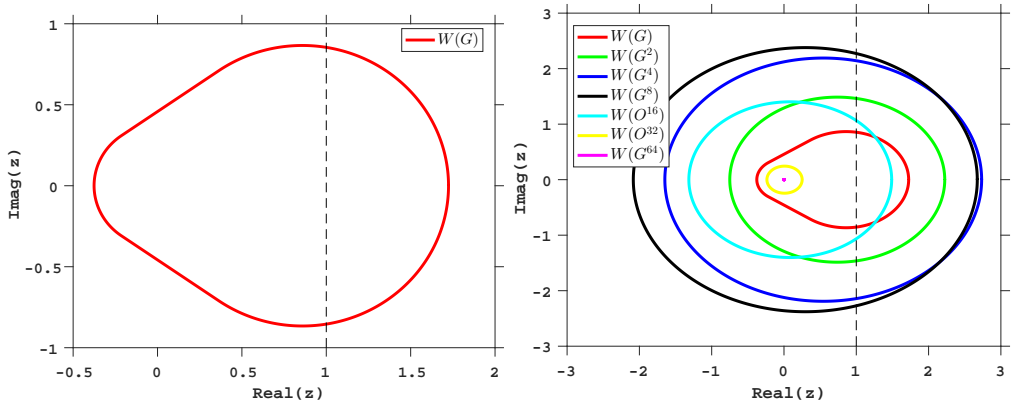

    \centering
    \includegraphics[width=0.40\textwidth]{figs/Random50_000e+00_Nesterov_FieldValues1.eps}
    \includegraphics[width=0.40\textwidth]{figs/Random50_000e+00_Nesterov_FieldValues.eps}
    \caption{Numerical range for the linear operator in Nesterov's method, on a random quadratic problem with dimension 50. Left: Operator~$G$. Right: Various operator powers~$G^p$. The RNA scheme will improve convergence whenever the point $(1,0)$ lies outside of the numerical range of the operator.}
    \label{fig:fieldvals_random50_nest}
\end{figure}

The difficulty in performing RNA on Nesterov's accelerated gradient method arises from the fact that the iterates can be non-monotonic. The restriction that $1$ should be outside the numerical range is necessary for both non-symmetric and symmetric operators. In symmetric operators, the numerical range is a line segment on the real axis and the numerical radius and spectral radius are equal, so this restriction is equivalent to having spectral radius less than $1$, i.e. having monotonically converging iterates.


\subsection{Chambolle-Pock's Primal-Dual Algorithm}
Chambolle-Pock is a first-order primal-dual algorithm used for minimizing composite functions of the form
\begin{equation}\label{prob:primal}
    \min_x h_p(x) := f(Ax) + g(x)
\end{equation}
where $f$ and $g$ are convex functions and $A$ is a continuous linear map. Optimization problems of this form arise in e.g. imaging applications like total variation minimization (see \cite{chambolle2016introduction}). The Fenchel dual of this problem is given by
\begin{equation}\label{prob:dual}
    \max_y h_d(y) := - f^*(-y) - g^*(A^*y)
\end{equation}
where $f^*, g^*$ are the convex conjugate functions of $f, g$ respectively. These problems are primal dual formulations of the general saddle point problem,
\begin{equation}\label{prob:saddle}
    \min_x \max_y <Ax, y> + g(x) - f^*(y),
\end{equation}
where $f^*, g$ are closed proper functions.
\cite{chambolle2011first} designed a first-order primal-dual algorithm for solving these problems, where primal-dual iterates are given by
\begin{equation}\label{eq:iter_cp}
\left\{
\begin{aligned}
    y_{k+1} &= \mathbf{Prox}_{f^*}^\sigma(y_k + \sigma A\bar{x}_k) \\
    x_{k+1} &= \mathbf{Prox}_g^\tau(x_k - \tau A^*y_{k+1}) \\
    \bar{x}_{k+1} &= x_{k+1} + \theta (x_{k+1} - x_{k})
\end{aligned}
\right.
\end{equation}
where $\sigma, \tau$ are the step length parameters, $\theta \in [0,1]$ is the momentum parameter and the proximal mapping of a function $f$ is defined as 
\[
    \mathbf{Prox}_f^\tau(y) = \arg \min_x \left\{\|y - x\|^2/({2\tau}) + f(x)\right\}
\]
Note that if the proximal mapping of a function is available then the proximal mapping of the conjugate of the function can be easily computed using Moreau's identity, with
\[
    \mathbf{Prox}_f^\tau(y) + \mathbf{Prox}_{f^*}^{1/\tau}(y/\tau) = y
\]
The optimal strategy for choosing the step length parameters $\sigma, \tau$ and the momentum parameter $\theta$ depend on the smoothness and strong convexity parameters of the problem. When $f^*$ and $g$ are strongly convex with strong convexity parameters $\delta$ and $\gamma$ respectively then these parameters are chosen to be constant values given as 
\begin{equation}\label{eq:pdgm_params}
    \sigma = \frac{1}{\|A\|}\sqrt{\frac{\gamma}{\delta}} \quad \tau = \frac{1}{\|A\|}\sqrt{\frac{\delta}{\gamma}} \qquad \theta = \left(1 + \frac{2\sqrt{\gamma\delta}}{\|A\|}\right)^{-1} 
\end{equation}
to yield the optimal linear rate of convergence. When only one of $f^*$ or $g$ is strongly convex with strong convexity parameter $\gamma$, then these parameters are chosen adaptively at each iteration as 
\begin{equation}
    \theta_{k} = (1 + 2\gamma\tau_k)^{-1/2}\quad\sigma_{k+1} ={\sigma_k}/{\theta_k} \quad \tau_{k+1} = \tau
_{k}\theta_k\end{equation}
to yield the optimal sublinear rate of convergence.

A special case of the primal-dual algorithm with no momentum term, i.e., $\theta = 0$ in \eqref{eq:iter_cp} is also known as the Arrow-Hurwicz method (\cite{arrowhurwicz1960}). Although theoretical complexity bounds for this algorithm are worse compared to methods including a momentum term, it is observed experimentally that the performance is either on par or sometimes better, when step length parameters are chosen as above.    

We first consider algorithms with no momentum term and apply RNA to the primal-dual sequence $z_k = (y_k,x_k)$. We note that, as observed in the Nesterov's case, RNA can only be applied on non-symmetric operators for which the normalization constant $1$ is outside their numerical range. Therefore, the step length parameters $\tau, \sigma$ should be suitably chosen such that this condition is satisfied. 

\subsubsection{Chambolle-Pock's Operator in the Quadratic Case}
When minimizing smooth strongly convex quadratic functions where $f(Ax) = \frac{1}{2}\|Ax - b\|^2$ and $g(x) = \frac{\mu}{2}\|x\|^2$, the proximal operators have closed form solutions. That is 
\[
\mathbf{Prox}_{f^*}^\sigma(y) =\frac{y - \sigma b}{1 + \sigma}
\quad \mbox{and}\quad  
\mathbf{Prox}_{g}^\tau(x) = \frac{1}{1 + \tau\mu}.
\]
Iterates of the primal-dual algorithm with no momentum term can be written as,
\begin{equation*}
\begin{aligned}
y_{k+1} &= \frac{y_k + \sigma Ax_k - \sigma b}{1 + \sigma},
\quad
x_{k+1} &= \frac{x_k - \tau A^Ty_{k+1}}{1 + \tau\mu}
\end{aligned}
\end{equation*}
Note that the optimal primal and dual solutions satisfy $y^* = Ax^* - b$ and $x^* = \frac{-1}{\mu}A^Ty$. This yields the following operator for iterations 
\begin{align}
G = \begin{bmatrix}
\frac{I}{1 + \sigma} & \frac{\sigma A}{1 + \sigma}\\
\frac{\tau A^T}{(1 + \sigma)(1 + \tau\mu)} &  \frac{I}{1 + \tau\mu} - \frac{\tau\sigma A^TA}{(1 + \sigma)(1 + \tau\mu)}. \\
\end{bmatrix}
\end{align}
Note that $G$ is a non-symmetric operator except when $\sigma = \frac{\tau}{1 + \tau\mu}$, in which case the numerical range is a line segment on the real axis and the spectral radius is equal to the numerical radius.

\subsubsection{Numerical Range}
The numerical range of the operator can be computed using the techniques described in Section \ref{s:nacc}. As mentioned earlier, the point $1$ should be outside the numerical range for the Chebyshev polynomial to be bounded. Therefore, using \eqref{eq: maxrealval}, we have,
$
    re(G) = \max Re(W(G))  = \lambda_{max}\left(\frac{G + G^*}{2}\right)
$
The step length parameters $\sigma, \tau$ should be chosen such that the above condition is satisfied. We observe empirically that there exists a range of values for the step length parameters such that $re(G) < 1$. Figure~\ref{fig:fieldvals_sonar} shows the numerical range of operator $G$ for $\sigma = 4, \tau =1/\|A^TA\|$ with two different regularization constants and Figure \ref{fig:fieldvals_sonar_contours} shows the regions for which $re(G^p)\leq 1$ (converging) for different values of $\sigma$ and~ $\tau$.

\begin{figure}[h!t]
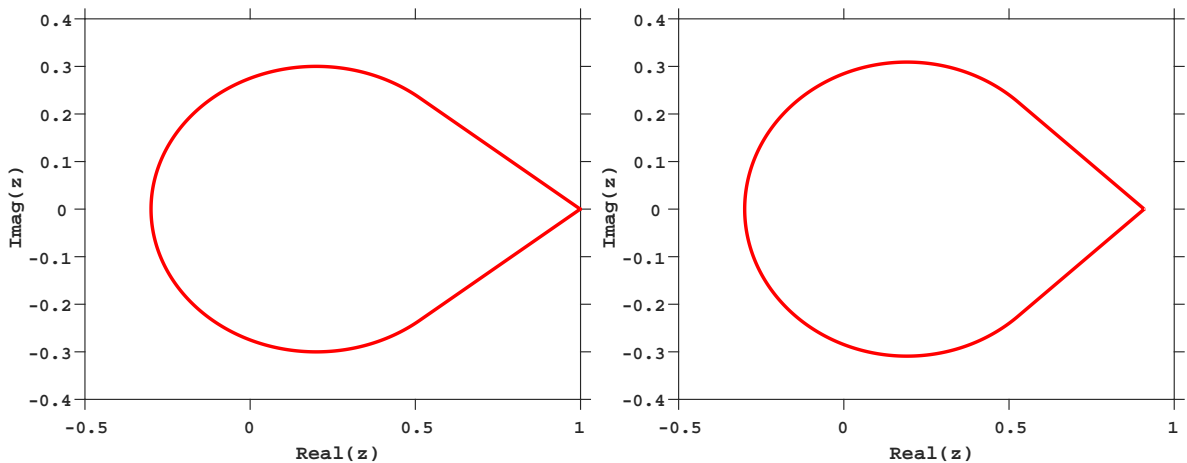

    \centering
    \includegraphics[width=0.47\textwidth]{figs/SonarNormalizedScaled_100e-03_CP_FieldValues_400e+00_100e+00.eps}
    \includegraphics[width=0.47\textwidth]{figs/SonarNormalizedScaled_100e-01_CP_FieldValues_400e+00_100e+00.eps}
    \caption{Field values for the Sonar dataset \citep{gorman1988analysis} with $\sigma = 4, \tau =1/\|A^TA\|$. The dataset has been scaled such that $\|A^TA\| = 1$. Left: $\mu = 10^{-3}$, right: $\mu = 10^{-1}$. The smaller numerical range on the right means faster convergence.}
    \label{fig:fieldvals_sonar}
\end{figure}

\begin{figure}[h!t]
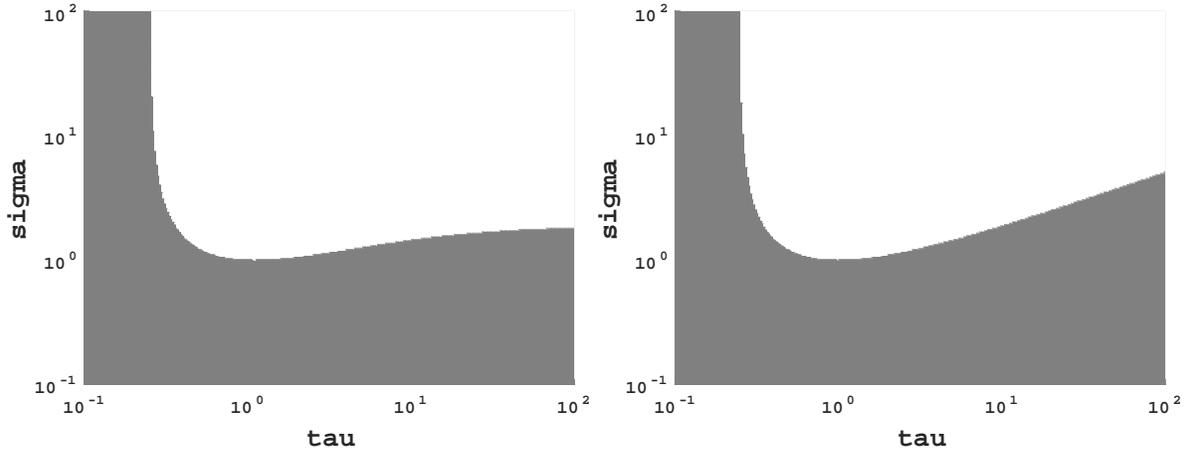

    \centering
    \includegraphics[width=0.47\textwidth]{figs/SonarNormalized_1_00e-01_CP_FieldValues_Imagesc_gray_v2.eps}
    \includegraphics[width=0.47\textwidth]{figs/SonarNormalized_1_00e-03_CP_FieldValues_Imagesc_gray_v2.eps}
    \caption{Plot of the $re(G^p)=1$ frontier with degree $p=5$ for the Sonar dataset \citep{gorman1988analysis} for different values of~$\tau$ and~$\sigma$. White color represents values for which $re(G^p)\leq 1$ (converging) and black color represents values $re(G^p)>1$ (not converging). Left: $\mu = 10^{-3}$. Right: $\mu = 10^{-1}$.}
    \label{fig:fieldvals_sonar_contours}
\end{figure}


%% file: perturbation_analysis.tex
\section{RNA on nonlinear iterations}\label{s:nonlin}

In previous sections, we analyzed the rate of convergence of RNA on linear algorithms (or quadratic optimization problems). In practice however, the operator $g$ is not linear, but can instead be nonlinear with potentially random perturbation. In this situation, regularizing parameter ensures RNA converges \citep{scieur2016regularized}. 

In this section, we first introduce the CNA algorithm, a constrained version of RNA that explicitly bounds the norm of the coefficients $c$ for the linear combinations. We show its equivalence with the RNA algorithm. Then, we analyze the rate of convergence of CNA when $g$ is a linear function perturbed with arbitrary errors, whose origin can be nonlinearities and/or random noises.

\subsection{Constrained Nonlinear Acceleration}
We now introduce the constrained version of RNA, replacing the regularization term by the hard constraint
\[
    \| c \|_2 \leq \frac{1+\tau}{\sqrt{N}}.
\]
In this algorithm, the parameter $\tau > 0$ controls the norm of the coefficients $c$. Of course, all the previous analysis applies to CNA, as RNA with $\lambda = 0$ is exactly CNA with $\tau = \infty$.

\begin{algorithm}[htb]
   \caption{Constrained Nonlinear Acceleration (\textbf{CNA})}
    \label{algo:cna}
\begin{algorithmic}
   \STATE {\bfseries Data:} Matrices $X$ and $Y$ of size $d\times N$ constructed from the iterates as in~\eqref{eq:general_iteration} and~\eqref{eq:def_xy}.
   \STATE {\bfseries Parameters:} Mixing $\eta\neq 0$, constraint $\tau \geq 0$.\\
   \hrulefill
   \STATE \textbf{1.} Compute matrix of residuals $R = X-Y$.
   \STATE \textbf{2.} Solve
   \BEQ
        \textstyle  c^{(\tau)} = \argmin_{c:c^T\textbf{1} = 1} \|Rc\|_2 \quad \text{s.t. } \; \|c\|_2\leq {\textstyle \frac{1+\tau}{\sqrt{N}}} \label{eq:ctau}
    \EEQ
    \STATE \textbf{3.} Compute extrapolated solution $\yex = (Y-\eta R)c^{(\tau)}$.
\end{algorithmic}
\end{algorithm}

\subsection{Equivalence Between Constrained \& Regularized Nonlinear Acceleration}

The parameters $\lambda$ in Algorithm \ref{algo:rna} and $\tau$ in Algorithm \ref{algo:cna} play similar roles. High values of $\lambda$ give coefficients close to simple averaging, and $\lambda = 0$ retrieves Anderson Acceleration. We have the same behavior when $\tau = 0$ or $\tau = \infty$. We can jump from one algorithm to the other using dual variables, since~\eqref{eq:cl} is the Lagrangian relaxation of the convex problem \eqref{eq:ctau}. This means that, for all values of $\tau$ there exists $\lambda = \lambda(\tau)$ that achieves $c^{\lambda} = c^{(\tau)}$. In fact, we can retrieve $\tau$ from the solution $c^{\lambda}$ by solving
\[
    \textstyle \frac{1+\tau}{\sqrt{N}} = \|c^{\lambda}\|_2.
\]
Conversely, to retrieve $\lambda$ from $c^{(\tau)}$, it suffices to solve
\BEQ
    \left\| \frac{(R^TR+(\lambda\|R\|^2_2) I)^{-1} \textbf{1}_N}{\textbf{1}_N^T(R^TR+(\lambda\|R\|^2_2) I)^{-1}\textbf{1}_N} \right\|^2 = \frac{(1+\tau)^2}{N}, \label{eq:nonlinear_equation}
\EEQ
assuming the constraint in \eqref{eq:ctau} tight, otherwise $\lambda = 0$. Because the norm in \eqref{eq:nonlinear_equation} is increasing with $\lambda$, a binary search or one-dimensional Newton methods gives the solution in a few iterations.

The next proposition bounds the norm of the coefficients of Algorithm~\ref{algo:rna} with an expression similar to~\eqref{eq:ctau}.
\begin{proposition}
    The norm of $c^{\lambda}$ from \eqref{eq:cl} is bounded by
    \BEQ
        \textstyle \|c^{\lambda}\|_2 \leq \frac{1}{\sqrt{N}}\sqrt{1+\frac{1}{\lambda}} \label{eq:bound_norm_lambda}
    \EEQ
\end{proposition}
\begin{proof}
    See \citet{scieur2016regularized}, (Proposition 3.2).
\end{proof}

Having established the equivalence between constrained and regularized nonlinear acceleration, the next section discusses the rate of convergence of CNA in the presence of perturbations.

\subsection{Constrained Chebyshev Polynomial}
The previous results consider the special cases where $\lambda = 0$ or $\tau = 0$,  which means that $\|c\|$ is unbounded. However, \citet{scieur2016regularized} show instability issues when $\|c\|$ is not controlled. Regularization is thus required in practice to make the method more robust to perturbations, even in the quadratic case (e.g., round-off errors). Unfortunately, this section will show that robustness comes at the cost of a potentially slower rate of convergence.

We first introduce \textit{constrained Chebyshev polynomials} for the range of a specific matrix. Earlier work in \citep{scieur2016regularized} considered regularized Chebyshev polynomials, but using a constrained formulation significantly simplifies the convergence analysis here. This polynomial plays an important role in Section~\ref{sec:convergence_rate_cna} in the convergence analysis.

\begin{definition}
The {Constrained Chebyshev Polynomial} $\chebpoly_N(x)$ of degree $N$ solves, for $\tau \geq 0$,
\BEQ
    \chebpoly_N(x) \triangleq \argmin_{p\in \pnp} \max_{x\in W(G)} p(x)  \quad \text{s.t.}  ~  \|p\|_2 \leq {\textstyle \frac{1+\tau}{\sqrt{1+N}}} \label{eq:constrained_cheby} 
\EEQ
in the variable ${p\in \pnp}$, where $W(G)$ is the numerical range of $G$. We write $\cheb_N \triangleq \| \chebpoly_N(G) \|_2$ the norm of the polynomial $\chebpoly_N$ applied to the matrix $G$.
\end{definition}

\subsection{Convergence Rate of CNA without perturbations} \label{sec:convergence_rate_cna}
The previous section introduced constrained Chebyshev polynomials, which play an essential role in our convergence results when $g$ is nonlinear and/or iterates~\eqref{eq:general_iteration} are noisy. Instead of analyzing Algorithm \ref{algo:rna} directly, we focus on its constrained counterpart, Algorithm \ref{algo:cna}. 

\begin{proposition} \label{prop:rate_constrained}
Let $X$, $Y$ \eqref{eq:def_xy} be build using iterates from \eqref{eq:general_iteration} where $g$ is linear \eqref{eq:linear_g} does not have $1$ as eigenvalue. Then, the norm of the residual \eqref{eq:residue} of the extrapolation produced by Algorithm \ref{algo:cna} is bounded by
\BEQ
    \|r(\yex)\|_2 \leq \| I-\eta (G-I)  \|_2 \|r(x_0)  \|_2\; \cheb_{N-1},
\EEQ
where $\tau \geq 0$ and $\cheb_N$ is defined in \eqref{eq:constrained_cheby}.
\end{proposition}
\begin{proof}
    The proof is similar to the one of Theorem \ref{thm:optimal_rate}. It suffices to use the constrained Chebyshev polynomial rather than the rescaled Chebyshev polynomial from \cite{golub1961chebyshev}.
\end{proof}

Proposition \ref{prop:rate_constrained} with $\tau = \infty$ gives the same result than Theorem \ref{thm:optimal_rate}. However, smaller values of $\tau$ give weaker results as $\cheb_{N-1}$ increases. However, smaller values of $\tau$ also reduce the norm of coefficients $c^{(\tau)}$ \eqref{eq:ctau}, which makes the algorithm more robust to noise. 

Using the constrained algorithm in the context of non-perturbed linear function $g$ yields no theoretical benefit, but the bounds on the extrapolated coefficients simplify the analysis of perturbed non-linear optimization schemes as we will see below.

In this section, we analyze the convergence rate of Algorithm \ref{algo:cna} for simplicity, but the results also hold for Algorithm \ref{algo:rna}. We first introduce the concept of perturbed linear iteration, then we analyze the convergence rate of RNA in this setting.

\textbf{Perturbed Linear Iterations.} 
Consider the following perturbed scheme,
\BEA \label{eq:perturbed_iteration_matrix}
    \tilde X_i = X^* + G(\tilde Y_{i-1}-X^*) + E_i, \qquad \tilde Y_i = [x_0,\tilde X_i] L_i,
\EEA
where $\tilde X_i$ and $\tilde Y_i$ are formed as in~\eqref{eq:def_xy} using the perturbed iterates $\tilde x_i$ and $\tilde y_i$, and $L_i$ is constructed using \eqref{eq:recurence_L}, and we write $E_i = [e_1,e_2,\ldots,e_i]$. For now, we do not assume anything on $e_i$ or $E_i$. This class contains many schemes such as gradient descent on nonlinear functions, stochastic gradient descent or even Nesterov's fast gradient with backtracking line search for example.

The notation \eqref{eq:perturbed_iteration_matrix} makes the analysis simpler than in \citep{scieur2016regularized,scieur2017nonlinear}, as we have the explicit form of the error over time. Consider the perturbation matrix $P_i$,
\BEQ
    P_i \triangleq \tilde R_i - R_i, \label{eq:perturbation_matrix}
\EEQ
Proposition \ref{prop:explicit_formula_perturbation} shows that the magnitude of the perturbations $\|P_i\|$ is proportional to the noise matrix $\|E_i\|$, i.e., $\|P_i\| = O(\|E_i\|)$.
\begin{proposition} \label{prop:explicit_formula_perturbation}
Let $P_i$ be defined in \eqref{eq:perturbation_matrix}, where $(\tilde X_i, \tilde Y_i)$ and $(\tilde X_i, \tilde Y_i)$ are formed respectively by \eqref{eq:general_iteration_matrix} and \eqref{eq:perturbed_iteration_matrix}. Let $\bar L_{j} = \| L_1\|_2  \| L_{2}\|_2 \ldots  \|L_j \|_2$. Then, we have the following bound
    \[
        \|P_i\| \leq 2\|E_i\| \bar L_{i} \sum_{j=1}^{i} \|G\|^j.
    \]
\end{proposition}
\begin{proof}
    First, we start with the definition of $R$ and $\tilde R$. Indeed,
    \[
        \tilde R_i - R_i = \tilde X_i-X_i - (\tilde Y_{i-1} - Y_{i-1}). 
    \]
    By definition,
    \[
        \tilde X_i - X_i = G(\tilde Y_{i-1}-X^*) + X^* + E_i -  G( Y_{i-1}-X^*) - X^* = G(\tilde Y_{i-1}-Y_{i-1}) + E_i
    \]
    On the other side,
    \[
        \tilde Y_{i-1} - Y_{i-1} = [0;\tilde X_{i-1}-X_{i-1}]L_{i-1}
    \]
    We thus have
    \BEAS
        P_i & =  & \tilde X_i-X_i - (\tilde Y_{i-1} - Y_{i-1}),\\
        & = & G(\tilde Y_{i-1}-Y_{i-1}) + E_i - [0;\tilde X_{i-1}-X_{i-1}]L_{i-1},\\
        & = & G( [0;\tilde X_{i-1}-X_{i-1}]L_{i-1}) + E_i - [0;G(\tilde Y_{i-2}-Y_{i-2}) + E_{i-1}]L_{i-1},\\
        & = & G [0;P_{i-1}]L_{i-1} + E_i - [0;E_{i-1}]L_{i-1}.\\
    \EEAS
    Finally, knowing that $\|E_i\| \geq \|E_{i-1}\|$ and $\|L_i\|\geq 1$, we expand
    \BEAS
        \|P_i\| & = & \| G \| \|P_{i-1}\|\|L_{i-1}\| + \|E_i\| + \|E_{i-1}\|\|L_{i-1}\|\\
        & \leq & \| G \| \|P_{i-1}\|\|L_{i-1}\| + 2\|E_i\| \|L_{i-1}\|
    \EEAS
    to have the desired result.
\end{proof}

We now analyze how close the output of Algorithm \ref{algo:cna} is to $x^*$. To do so, we compare scheme \eqref{eq:perturbed_iteration_matrix} to its perturbation-free counterpart \eqref{eq:general_iteration_matrix}. Both schemes have the same starting point $x_0$ and ``fixed point''~$x^*$. It is important to note that scheme~\eqref{eq:perturbed_iteration_matrix} may not converge due to noise. The next theorem bounds the accuracy of the output of CNA.
\begin{theorem} \label{thm:convergence_perturbation}
    Let $\yex$ be the output of Algorithm \eqref{algo:cna} applied to \eqref{eq:perturbed_iteration_matrix}. Its accuracy is bounded by
    \BEAS
        \|(G-I) \left(\yex - x^*\right)\|
        \leq \|I-\eta(G-I) \| \Big(\underbrace{ \cheb_{N-1} \|(G-I)(x_0-x^*)\|}_{\textbf{acceleration}}
        + \underbrace{\textstyle \frac{1+\tau}{\sqrt{N}} \big( \|P_N\| + \|E_N\|\big)}_{\textbf{stability}}\Big).
    \EEAS
\end{theorem}
\begin{proof}
We start with the following expression for arbitrary coefficients $c$ that sum to one,
\[
    (G-I) \left((\tilde Y - \eta \tilde R)c - x^*\right).
\]
Since
\[
    \tilde R = \tilde X - \tilde Y = (G-I)(\tilde Y - X^*) + E,
\]
we have
\[
    (G-I)(\tilde Y-X^*)  = (\tilde R-E).
\]
So,
\[
    (G-I) (\tilde Y-X^* - \eta \tilde R)c = (\tilde R-E)c - \eta (G-I)\tilde Rc .
\]
After rearranging the terms we get
\BEQ
(G-I) \left((\tilde Y - \eta \tilde R)c - x^*\right) = (I-\eta(G-I))\tilde Rc - E c.\label{eq:decomposition_error}
\EEQ
We bound \eqref{eq:decomposition_error} as follow, using coefficients from \eqref{eq:ctau},
\[
    \|I-\eta(G-I)\| \|\tilde Rc^{(\tau)}\| + \|E\| \|c^{(\tau)}\|.
\]
Indeed,
\BEAS
    \|\tilde Rc^{(\tau)}\|^2 & = & \min_{c: c^T \textbf{1} = 1,\; \|c\| \leq \frac{1+\tau}{\sqrt{N}}}  \|\tilde Rc\|^2.
\EEAS
We have
\BEAS
    \min_{c:\textbf{1}^Tc=1,\; \|c\| \leq \frac{1+\tau}{\sqrt{N}}} \|\tilde Rc\|_2,  & \leq & \min_{c:\textbf{1}^Tc=1\; \|c\| \leq \frac{1+\tau}{\sqrt{N}}} \|Rc\|_2 + \|P_Rc\|_2,\\
    & \leq & \left(\min_{c:\textbf{1}^Tc=1\; \|c\| \leq \frac{1+\tau}{\sqrt{N}}} \|Rc\|_2\right) + \|P_R\|_2\frac{1+\tau}{\sqrt{N}} ,\\
    & \leq & \cheb_{N-1} \|r(x_0)\| +  \frac{\|P_R\|(1+\tau)}{\sqrt{N}}.
\EEAS
This prove the desired result.
\end{proof}

This theorem shows that Algorithm \ref{algo:cna} balances acceleration and robustness. The result bounds the accuracy by the sum of an \textit{acceleration term} bounded using constrained Chebyshev polynomials, and a \textit{stability} term proportional to the norm of perturbations. In the next section, we consider the particular case where $g$ corresponds to a gradient step when the perturbations are Gaussian or due to non-linearities.

\section{Convergence Rates for CNA on Gradient Descent}\label{s:grad}
We now apply our results when $g$ in \eqref{eq:general_iteration_matrix} corresponds to the gradient step
\BEQ
    x-h\nabla f(x), \label{eq:gradient_step_g}
\EEQ
where $f$ is the objective function and $h$ a step size. We assume the function $f$ twice differentiable, $L$-smooth and $\mu$-strongly convex. This means
\BEQ
    \mu I \leq \nabla^2 f(x) \leq LI. \label{eq:smooth_strong_convex}
\EEQ
Also, we assume $h = \frac{1}{L}$ for simplicity. Since we consider optimization of differentiable functions here, the matrix $\nabla^2 f(x^*)$ is symmetric.

When we apply the gradient method \eqref{eq:gradient_step_g}, we first consider its linear approximation
\BEQ
    g(x) = x-h\nabla^2 f(x^*) (x-x^*).
    \label{eq:linear_gradient_step}
\EEQ
with stepsize $h=1/L$. We identify the matrix $G$ in \eqref{eq:linear_g} to be
\[
    G = I-\frac{\nabla^2 f(x^*)}{L}.
\]
In this case, and because the Hessian is now symmetric, the numerical range $W(G)$ simplifies into the line segment
\[
    W(G) = [0,1-\kappa],
\]
where $\kappa = \frac{\mu}{L} < 1$ often refers to the inverse of the condition number of the matrix $\nabla^2 f(x^*)$.

In the next sections, we study two different cases. First, we assume the objective quadratic, but \eqref{eq:linear_gradient_step} is corrupted by a random noise. Then, we consider a general nonlinear function $f$, with the additional assumption that its Hessian is Lipchitz-continuous. This corresponds to a nonlinear, deterministic perturbation of \eqref{eq:linear_gradient_step}, whose noise is bounded by $O(\|x-x^*\|^2)$.

\subsection{Random Perturbations}

We perform a gradient step on the quadratic form
\[
    f(x) = \frac{1}{2}(x-x^*) A (x-x^*), \;\; \mu I \preceq A \preceq L I.
\]
This corresponds to \eqref{eq:linear_gradient_step} with $\nabla f(x^*) = A$. However, each iteration is corrupted by $e_i$, where $e_i$ is Gaussian with variance $\sigma^2$. The next proposition is the application of Theorem \ref{thm:convergence_perturbation} to our setting. To simplify results, we consider $\eta = 1$.
\begin{proposition} \label{prop:convergence_stoch_gradient}
Assume we use Algorithm \eqref{algo:cna} with \mbox{$\eta = 1$} on $N$ iterates from \eqref{eq:perturbed_iteration_matrix}, where $g$ is the gradient step \eqref{eq:gradient_step_g} and $e_i$ are zero-mean independent random noise with variance bounded by $\sigma^2$. Then,
\BEQ
    \mathbb{E}[\|\nabla f(\yex)\|] \leq (1-\kappa) \;\cheb_{N-1}\|\nabla  f(x_0)\| + \mathcal{E} ,
\EEQ
where
\[
    \mathcal{E} \leq  (1-\kappa)\frac{1+\tau}{\sqrt{N}} L\sigma \sum_{j=1}^{N} (1-\kappa)^j \bar L_{j}.
\]
In the simple case where we accelerate the gradient descent algorithm, all $L_i = I$ and thus
\[
    \textstyle \mathcal{E} \leq  \frac{1+\tau}{\sqrt{N}} \frac{L\sigma}{\kappa}. 
\]
\end{proposition}

\begin{proof}
    Since $\eta = 1$,
    \[
        \|I-\eta(G-I)\| = \|G\| \leq 1-\kappa.
    \]
    Now, consider $\mathbb{E}[\|E\|]$. Because each $e_i$ are independents Gaussian noise with variance bounded by $\sigma$, we have,
    \[
        \mathbb{E}[\|E\|] \leq \sqrt{\mathbb{E}[\|E\|^2]} \leq \sigma.
    \]
    Similarly, for $P$ \eqref{eq:perturbation_matrix}, we use Proposition \eqref{prop:explicit_formula_perturbation} and we have
    \BEAS
        \mathbb{E}[\|P\|] & \leq & \textstyle \mathbb{E}[\|E_i\|] \left(1+\sum_{j=1}^{i} (1-\kappa)^j \bar L_{j}\right)\\
        & \leq & \textstyle  \sigma  \left(1+\sum_{j=1}^{i} (1-\kappa)^j \bar L_{j}\right)
    \EEAS
    Thus, $\mathcal{E}_{N}^{\kappa,\tau}$ in Theorem \ref{thm:convergence_perturbation} becomes
    \[
        \mathcal{E}_{N}^{\kappa,\tau} \leq \textstyle  \frac{\sigma(1+\tau)}{\sqrt{N}}  \left(2+\sum_{j=1}^{N} (1-\kappa)^j \bar L_{j}\right)
    \]
    Finally, it suffice to see that
    \[
        (G-I)(x-x^*)+x^* = (A/L)(x-x^*)+x^* = \frac{1}{L} \nabla f(x),
    \]
    and we get the desired result. In the special case of plain gradient method, $L_i = I$ so $\bar L_i = 1$. We then get
    \[
        \textstyle \sum_{j=1}^{N} (1-\kappa)^j  \leq \sum_{j=1}^{\infty} (1-\kappa)^j \leq \frac{1}{\kappa}.
    \]
    which is the desired result.
\end{proof}

This proposition also applies to gradient descent with momentum or with our online acceleration algorithm \eqref{eq:online_rna}. We can distinguish two different regimes when accelerating gradient descent with noise. One when $\sigma$ is small compared to $\|f(x_0)\|$, and one when $\sigma$ is large. In the first case, the acceleration term dominates. In this case, Algorithm \ref{algo:cna} with large $\tau$ produces output $\yex$ that converges with a near-optimal rate of convergence. In the second regime where the noise dominates, $\tau$ should be close to zero. In this case, using our extrapolation method when perturbation are high naturally gives the simple averaging scheme. We can thus see Algorithm \eqref{algo:cna} as a way to interpolate optimal acceleration with averaging.

\subsection{Nonlinear Perturbations}
Here, we study the general case where the perturbation $e_i$ are bounded by a function of $D$, where $D$ satisfies
\BEQ
    \| \tilde y_i - x^* \|_2 \leq D \qquad \forall i. \label{eq:def_d}
\EEQ
This assumption is usually met when we accelerate non-divergent algorithms. More precisely, we assume the perturbation are bounded by
\BEQ
    \big(\|I-\eta(G-I) \| \|P_N\| + \|E\|\big) \leq \gamma\sqrt{N} D^\alpha. \label{nonlinear_perturbation}
\EEQ
where $\gamma$ and $\alpha$ are scalar. Since $\|P_N\| = O(\|E\|)$ by proposition \ref{prop:explicit_formula_perturbation}, we have that
\BEQ \label{eq:condition_perturbation_column}
    \|e_i\| \leq O(D^\alpha) \Rightarrow \eqref{nonlinear_perturbation}.
\EEQ
We call these perturbations "nonlinear" because the error term typically corresponds to the difference between $g$ and its linearization around $x^*$. For example, the optimization of smooth non-quadratic functions with gradient descent can be described using \eqref{nonlinear_perturbation} with $\alpha = 1$ or $\alpha = 2$, as shown in Section \ref{sec:smooth_functions}. The next proposition bounds the accuracy of the extrapolation produced by Algorithm \eqref{algo:cna} in the presence of such perturbation.
\begin{proposition} \label{prop:conv_nonlinear}
    Consider Algorithm \eqref{algo:cna} with $\eta = 1$ on $N$ iterates from \eqref{eq:perturbed_iteration_matrix}, where perturbations satisfy \eqref{eq:def_d}. Then,
    \BEAS
        \textstyle \left\|(G-I)(\yex-x^*)\right\|
        & \leq &  (1-\kappa)\Big(  \cheb_{N-1}\left\|(G-I)(x_0-x^*)\right\| + \mathcal{E}\Big)
    \EEAS
    where $\mathcal{E} \leq (1+\tau)\gamma D^\alpha$.
\end{proposition}
\begin{proof}
    Combine Theorem \ref{thm:convergence_perturbation} with assumption \eqref{nonlinear_perturbation}.
\end{proof}

Here, $\|x_0-x^*\|$ is of the order of $D$. This bound is generic as does not consider any strong structural assumption on $g$, only that its first-order approximation error is bounded by a power of $D$. We did not even assume that scheme \eqref{eq:perturbed_iteration_matrix} converges. This explains why Proposition \ref{prop:conv_nonlinear} does not necessary give a convergent bound. Nevertheless, in the case of convergent scheme, Algorithm \ref{algo:cna} with $\tau = 0$ output the average of previous iterates, that also converge to $x^*$.

However, Proposition \ref{prop:conv_nonlinear} is interesting when perturbations are small compared to $\|x_0-x^*\|$. In particular, it is possible to link $\tau$ and $D^\alpha$ so that Algorithm \ref{algo:cna} asymptotically reach an optimal rate of convergence, when $D\rightarrow 0$.

\begin{proposition}\label{prop:asymptotic_optimal_rate}
    If $\tau = O(D^{-s})$ with $0<s<\alpha-1$, then, when $D\rightarrow 0$, Proposition \ref{prop:conv_nonlinear} becomes
    \BEAS
        \textstyle \left\|(G-I)(\yex-x^*)\right\|
        & \leq &  (1-\kappa)  \left(\frac{1-\sqrt{\kappa}}{1+\sqrt{\kappa}}\right)^{N-1}\left\|(G-I)(x_0-x^*)\right\|
    \EEAS
    The same result holds with Algorithm \ref{algo:rna} if $\lambda = O(D^r)$ with $0<r<2(\alpha-1)$.
\end{proposition}

\begin{proof}
By assumption, 
\[
    \|x_0 - x^*\| = O(D).
\]
We thus have, by Proposition \ref{prop:conv_nonlinear}
    \BEAS
        \textstyle \left\|(G-I)(\yex-x^*)\right\| 
        & \leq &  (1-\kappa)\Big(  \cheb_{N-1}O(D) + (1+\tau)O(D^{\alpha})\Big).
    \EEAS
    $\tau$ will be a function of $D$, in particular $\tau = D^{-s}$. We want to have the following conditions,
    \[
        \lim\limits_{D\rightarrow 0} (1+\tau(D))D^{\alpha-1} = 0, \qquad \lim\limits_{D\rightarrow 0} \tau = \inf.
    \]
    The first condition ensures that the perturbation converge faster to zero than the acceleration term. The second condition ask $\tau$ to grow as $D$ decreases, so that CNA becomes unconstrained. Since $\tau = D^{-s}$, we have to solve
    \[
        \lim\limits_{D\rightarrow 0} D^{\alpha-1} + D^{\alpha-s-1} = 0, \qquad \lim\limits_{D\rightarrow 0} D^{-s} = \inf.
    \]
    Clearly, $0 < s < \alpha-1$ satisfies the two conditions. After taking the limit, we obtain
    \[
        \textstyle \left\|(G-I)(\yex-x^*)\right\| 
        \leq  (1-\kappa) \cheb_{N-1} \|(G-I)(x_0-x^*)\|
    \]
    Since $W(G)$ is the real line segment $[0,1-\kappa]$, and because $\tau \rightarrow \infty$, we end with an unconstrained minimax polynomial. Therefore, we can use the result from \citet{golub1961chebyshev},
    \[
        \min_{p\in \pnp}\max_{\lambda \in [0,1-\kappa]} |p(\lambda)|  \leq \left(\frac{1-\sqrt{\kappa}}{1+\sqrt{\kappa}}\right)^{N-1}.
    \]
    
    For the second result, by using \eqref{eq:bound_norm_lambda},
    \[
        \|c^{\lambda}\|_2 \leq \frac{1}{\sqrt{N}}\sqrt{1+\frac{1}{\lambda}}.
    \]
    Setting 
    \[
        \frac{1+\tau}{\sqrt{N}} = \frac{1}{\sqrt{N}}\sqrt{1+\frac{1}{\lambda}}
    \]
    with $\tau = D^{-s}$ gives the conditions.
\end{proof}

This proposition shows that, when perturbations are of the order of $D^\alpha$ with $\alpha > 1$, then our extrapolation algorithm converges optimally once the $\tilde y_i$ are close to the solution $x^*$. The next section shows this holds, for example, when minimizing function with smooth gradients.

\subsection{Optimization of Smooth Functions} \label{sec:smooth_functions}
Let the objective function $f$ be a nonlinear function that follows \eqref{eq:smooth_strong_convex}, which also has a Lipchitz-continuous Hessian with constant $M$,
\BEQ \label{eq:smooth_gradient}
    \|\nabla^2 f(y)-\nabla^2 f(x)\| \leq M\|y-x\|.
\EEQ
This assumption is common in the convergence analysis of second-order methods. For the convergence analysis, we consider that $g(x)$ perform a gradient step on the quadratic function
\BEQ
    \frac{1}{2}(x-x^*)\nabla^2 f(x^*)(x-x^*). \label{eq:gradient_approx}
\EEQ
This is the quadratic approximation of $f$ around $x^*$. The gradient step thus reads, if we set $h=1/L$,
\BEQ
    g(x) = \left(I-\frac{\nabla^2 f(x^*)}{L}\right)(x-x^*)+x^*. \label{eq:gradient_step_linearized}
\EEQ

The perturbed scheme corresponds to the application of \eqref{eq:gradient_step_linearized} with a specific nonlinear perturbation,
\BEQ
    \textstyle \tilde x_{i+1} = g(\tilde y_i) - \underbrace{ \textstyle \frac{1}{L}(\nabla f(\tilde y_i)-\nabla^2 f(x^*)(\tilde y_i-x^*))}_{=e_i}. \label{eq:nonlinear_perturbation}
\EEQ
This way, we recover the gradient step on the non-quadratic function $f$. The next Proposition shows that schemes \eqref{eq:nonlinear_perturbation} satisfies \eqref{nonlinear_perturbation} with $\alpha = 1$ when $D$ is big, or $\alpha=2$ when $D$ is small.
\begin{proposition}\label{eq:bound_function_smooth_gradient}
    Consider the scheme \eqref{eq:nonlinear_perturbation}, where $f$ satisfies \eqref{eq:smooth_gradient}. If $\|y_i-x^*\| \leq D$, then \eqref{eq:condition_perturbation_column} holds with $\alpha = 1$ for large $D$ or $\alpha = 2$ for small $D$, i.e.,
    \[
        \|e_i\| = \|\frac{1}{L}(\nabla f(\tilde y_i)-\nabla^2 f(x^*)(\tilde y_i-x^*))\| \leq \min\{ \|y_i-x^*\| ,\;\;
    \frac{M}{2L} \|y_i-x ^*\|^2\} \leq \min\{ D ,\;\;
    \frac{M}{2L} D^2\}.
    \]
\end{proposition}
\begin{proof}
The proof of this statement can be found in \citet{nesterov2006cubic}.
\end{proof}

The combination of Proposition \ref{prop:asymptotic_optimal_rate} with Proposition \ref{eq:bound_function_smooth_gradient} means that RNA (or CNA) converges asymptotically when $\lambda$ (or $\tau$) are set properly. In other words, if $\lambda$ decreases a little bit faster than the perturbations, the extrapolation on the perturbed iterations behave as if it was accelerating a perturbation-free scheme. Our result improves that in \cite{scieur2016regularized,scieur2017nonlinear}, where $r\in]0,\frac{2(\alpha-1)}{3}[$.

\section{Online Acceleration} \label{s:online}
We now discuss the convergence of online acceleration, i.e. coupling iterates in $g$ with the extrapolation Algorithm~\ref{algo:rna} at each iteration when $\lambda = 0$. The iterates are now given by
\BEA \label{eq:online_rna}
    x_{N} = g(y_{N-1}),\qquad y_N = \textbf{RNA}(X,Y,\lambda,\eta),
\EEA
where $\textbf{RNA}(X,Y,\lambda,\eta)=\yex$ with $\yex$ the output of Algorithm~\ref{algo:rna}. By construction, $\yex$ is written
\[
    \textstyle \yex = \sum_{i=1}^N \cl_i (y_{i-1} - \eta (x_i-y_{i-1})).
\]
If $\cl_N\neq 0$ then $\yex$ matches \eqref{eq:general_iteration}, thus online acceleration iterates in~\eqref{eq:online_rna} belong to the class of algorithms in~\eqref{eq:general_iteration}. If we can ensure $\cl_N \neq 0$, applying Theorem~\ref{thm:optimal_rate} recursively will then show an optimal rate of convergence for online acceleration iterations in~\eqref{eq:online_rna}. We do this for linear iterations in what follows.

\subsection{Linear Iterations}
The next proposition shows that either $\cl_N \neq 0$ holds, or otherwise $\yex = x^*$ in the linear case.

\begin{proposition} \label{prop:online_accel_structure}
Let $X$, $Y$ \eqref{eq:def_xy} be built using iterates from \eqref{eq:general_iteration}. Let $g$ be defined in \eqref{eq:linear_g}, where the eigenvalues of $G$ are different from one. Consider $\yex$ the output of Algorithm \ref{algo:rna} with $\lambda = 0$ and $\eta \neq 0$. If $R = X-Y$ is full column rank, then $\cl_N \neq 0$. Otherwise, $\yex = x^*$.
\end{proposition}
\begin{proof}
    Since, by definition, $\textbf{1}^T c^{\lambda}=1$, it suffices to prove that the last coefficient $c^{\lambda}_N \neq 0$. For simplicity, in the scope of this proof we write $c=c^{\lambda}$ We prove it by contradiction. Let $R_-$ be the matrix $R$ without its last column, and $c_-$ be the coefficients computed by RNA using $R_-$. Assume $c_N = 0$. In this case,
    \[
        c = [c_-;\; 0] \qquad \text{and} \qquad Rc = R_- c_-.
    \]
    This also means that, using the explicit formula for $c$ in \eqref{eq:cl},
    \[
        \frac{(R^TR)^{-1}\textbf{1}}{\textbf{1}(R^TR)^{-1}\textbf{1}} = \left[ \frac{(R_-^TR_-)^{-1}\textbf{1}}{\textbf{1}(R_-^TR_-)^{-1}\textbf{1}};\; 0\right], \qquad \Leftrightarrow \qquad (R^TR)^{-1}\textbf{1} = \left[ (R_-^TR_-)^{-1}\textbf{1};\; 0\right].
    \]
    The equivalence is obtained because
    \[
        \textbf{1}(R^TR)^{-1}\textbf{1} = \textbf{1}^T c = \textbf{1}^T c_- = \textbf{1}(R_-^TR_-)^{-1}\textbf{1}.
    \]
    We can write $c$ and $c_-$ under the form of a linear system,
    \[
        R^TRc = \alpha \textbf{1}_N, \quad (R_-^TR_-)c_- = \alpha \textbf{1}_{N-1},
    \]
    where $\alpha  = \textbf{1}(R^TR)^{-1}\textbf{1} = \textbf{1}(R_-^TR_-)^{-1}\textbf{1}$, which is nonzero. We augment the system with $c_-$ by concatenating zeros,
    \[
        R^TRc = \alpha \textbf{1}_N, \quad 
        \begin{bmatrix}
            (R_-^TR_-) & 0_{N-1 \times 1} \\ 
            0_{1 \times N-1} & 0
        \end{bmatrix}
        \begin{bmatrix}
        c_- \\
        0
        \end{bmatrix}
        = \alpha
        \begin{bmatrix}
            \textbf{1}_{N-1}\\
            0
        \end{bmatrix}
    \]
    Let $r_+$ be the residual at iteration $N$. This means $R = [R_-,r_+]$. We substract the two linear systems,
    \[
        \begin{bmatrix}
            0 & R^Tr_+ \\ 
            r_+^TR & r_+^Tr_+
        \end{bmatrix}
        \begin{bmatrix}
        c_- \\
        0
        \end{bmatrix}
        =
        \begin{bmatrix}
            0\\
            \alpha \neq 0
        \end{bmatrix}
    \]
    The $N-1$ first equations tells us that either $(R^Tr_+)_i$ or $c_{-,i}$ are equal to zero. This implies 
    \[
        (R^Tr_+)^Tc = \sum_{i=1}^{N-1}(R^Tr_+)^T_ic_i=0.
    \]
    However, the last equation reads
    \[
        (R^Tr_+)^Tc + 0\cdot r_+^Tr_+ \neq 0.
    \]
    This is a contradiction, since 
    \[
        (R^Tr_+)^Tc + 0\cdot r_+^Tr_+ = 0.
    \]
    Now, assume $R$ is not full rank. This means there exist a non-zero linear combination such that
    \[
        Rc = 0.
    \]
    However, due to its structure $R$ is a Krylov basis of the Krylov subspace
    \[
        \mathcal{K}_N = \text{span}[r_0,Gr_0,\ldots, G^{N}]
    \]
    If the rank of $R$ is strictly less $N$ (says $N-1$), this means
    \[
        \mathcal{K}_N = \mathcal{K}_{N-1}.
    \]
    Due to properties of the Krylov subspace, this means that
    \[
        r_0 = \sum_{i=1}^{N-1} \alpha_i \lambda_i v_i
    \]
    where $\lambda_i$ are distinct eigenvalues of $G$, and $v_i$ the associated eigenvector. Thus, it suffices to take the polynomial $p^*$ that interpolates the $N-1$ distinct $\lambda_i$. In this case,
    \[
        p^*(G)r_0 = 0.
    \]
    Since $p(1)\neq 0$ because $\lambda_i \leq 1-\kappa < 1$, we have
    \[
        \min \|Rc\| = \min_{p\in \pnm} \|p(G)r_0\| = \frac{p^*(G)}{p(1)}r_0 = 0. 
    \]
    which is the desired result.
\end{proof}

This shows that we can use \textit{RNA to accelerate iterates coming from RNA}. In numerical experiments, we will see that this new approach significantly improves empirical performance.

\subsection{RNA \& Nesterov's Method}
We now briefly discuss a strategy that combines Nesterov's acceleration with RNA. This means using RNA instead of the classical momentum term in Nesterov's original algorithm.  Using RNA, we can produce iterates that are asymptotically adaptive to the problem constants, while ensuring an optimal upper bound if one provides constants $L$ and $\mu$. We show below how to design a condition that decides after each gradient steps if we should combine previous iterates using RNA or Nesterov coefficients.

Nesterov's algorithm first performs a gradient step, then combines the two previous iterates. A more generic version with a basic line search reads 
\BEQ
    \begin{cases}
        \text{Find } x_{i+1} : f(x_{i+1}) \leq f(y_i) - \frac{1}{2L}\|f(y_i)\|_2^2\\
        \textstyle y_{i+1}  = (1+\beta) x_{i+1} - \beta x_{i}, \quad \beta = \frac{1-\sqrt{\kappa}}{1+\sqrt{\kappa}}\,.
    \end{cases} \label{eq:general_nesterov_step}
\EEQ
The first condition is automatically met when we perform the gradient step $x_{i+1} = x_i - \nabla f(x_i)/L$. Based on this, we propose the following algorithm.

\begin{algorithm}[htb]
   \caption{Optimal Adaptive Algorithm}
    \label{algo:optimal_adaptive}
\begin{algorithmic}
    \STATE Compute gradient step $x_{i+1} = y_{i} - \frac{1}{L} \nabla f(y_i)$.
    \STATE Compute $\yex = \textbf{RNA}(X,Y,\lambda,\eta)$.
    \STATE Let
    \[
        z = \frac{\yex + \beta x_i}{1+\beta}
    \]
    \STATE Choose the next iterate, such that
    \[
        y_{i+1} =
        \begin{cases}
            \yex \quad \text{If}\;\; f(z) \leq f(x_i) - \frac{1}{2L}\|f(x_i)\|_2^2,\\
            (1+\eta) x_i - \eta x_{i-1}\quad \text{Otherwise}.
        \end{cases}
    \]
\end{algorithmic}
\end{algorithm}

Algorithm \ref{algo:optimal_adaptive} has an optimal rate of convergence, i.e., it preserves the worst case rate of the original Nesterov algorithm. The proof is straightforward: if we do not satisfy the condition, then we perform a standard Nesterov step ; otherwise, we pick $z$ instead of the gradient step, and we combine
\[
    y_{i+1} = (1+\eta) z - \eta x_{i-1} = \yex.
\]
By construction this satisfies~\eqref{eq:general_nesterov_step}, and inherits its properties, like an optimal rate of convergence.

%% file: numexp_short.tex
\section{Numerical Results}\label{s:numres}

We now study the performance of our techniques on  $\ell_2$-regularized logistic regression using acceleration on Nesterov's accelerated method\footnote{The source code for the numerical experiments can be found on GitHub at \url{https://github.com/windows7lover/RegularizedNonlinearAcceleration}}.

We solve a classical regression problem on the Madelon-UCI dataset \citep{guyon2003design} using the logistic loss with $\ell_2$ regularization. The regularization has been set such that the condition number of the function is equal to $10^{6}$. We compare to standard algorithms such as the simple gradient scheme, Nesterov's method for smooth and strongly convex objectives \citep{nesterov2013introductory} and L-BFGS. For the step length parameter, we used a backtracking line-search strategy. We compare these methods with their offline RNA accelerated counterparts, as well as with the online version of RNA described in~\eqref{eq:online_rna}. 

\begin{figure}[h!t]
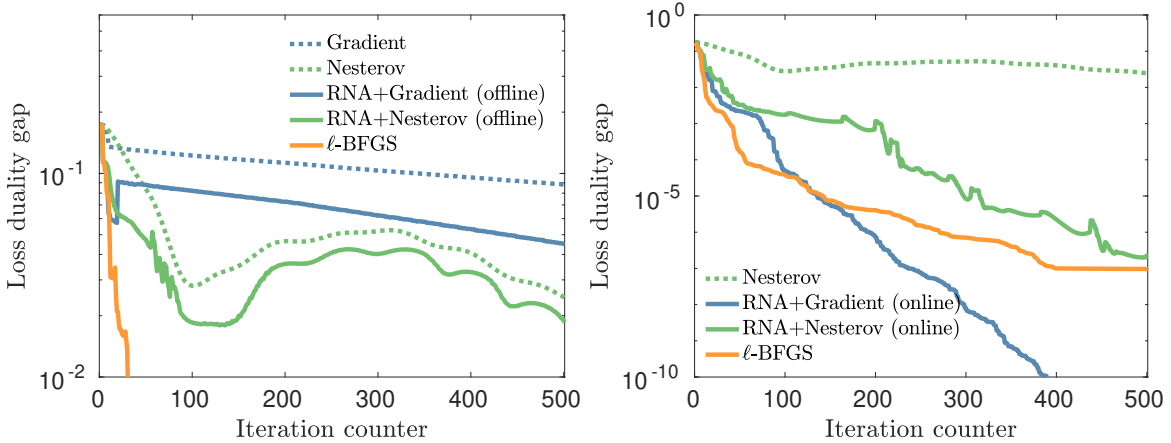

    \centering
    \includegraphics[width=0.47\textwidth]{figs/offline_madelon.eps}
    \includegraphics[width=0.47\textwidth]{figs/online_madelon.eps}
    \caption{Logistic loss on the Madelon \citep{guyon2003design}. Comparison between offline (\textit{left}) and online (\textit{right}) strategies for RNA on gradient and Nesterov's method. We use $\ell$-BFGS (with $\ell=100$ gradients stored in memory) as baseline. Clearly, one step of acceleration improves the accuracy. The performance of online RNA, which applies the extrapolation at \textit{each} step, is similar to that of L-BFGS methods, though RNA does not use line-search and requires 10 times less memory.}
    \label{fig:madelon}
\end{figure}
We observe in Figure \ref{fig:madelon} that offline RNA improves the convergence speed of gradient descent and Nesterov's method. However, the improvement is only a constant factor: the curves are shifted but have the same slope. Meanwhile, the online version greatly improves the rate of convergence, transforming the basic gradient method into an optimal algorithm competitive with line-search L-BFGS.

In contrast to most quasi-Newton methods (such as L-BFGS), RNA does \textit{not} require a Wolfe line-search to be convergent. This is because the algorithm is stabilized with a Tikhonov regularization. In addition, the regularization in a way controls the impact of the noise in the iterates, making the RNA algorithm suitable for stochastic iterations \citep{scieur2017nonlinear}.

We also tested the performance of online RNA on general non-symmetric algorithm, Primal-Dual Gradient Method (PDGM) \citep{chambolle2011first} defined in \eqref{eq:iter_cp} with $\theta=0$. We observe in Figure \ref{fig:madelon_nonsym_logistic1} that RNA has substantially improved the performance of the base algorithm.

\begin{figure}[ht]
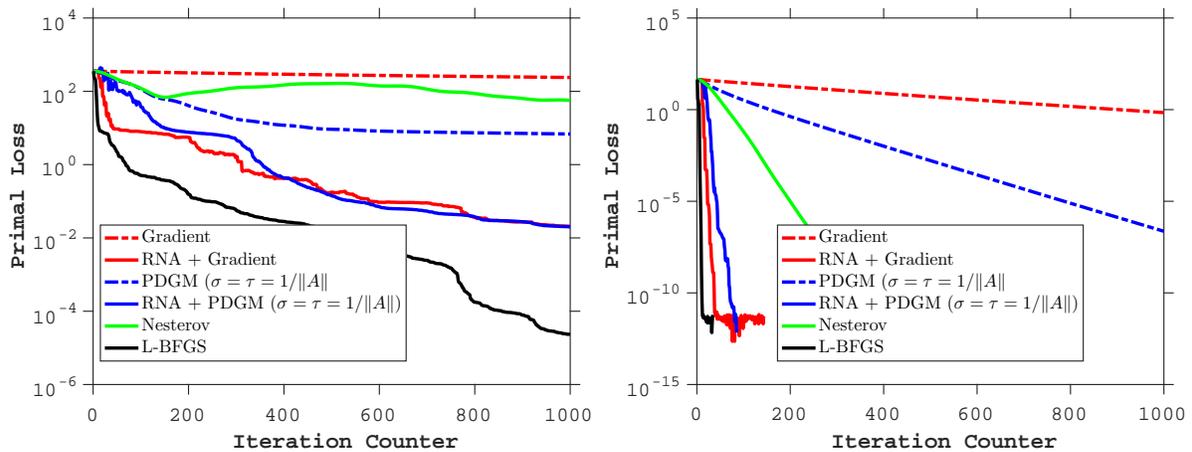

    \centering
    \includegraphics[width=0.47\textwidth]{figs/MadelonNormalized_logistic_100e-02_PrimalLoss_Journal.eps}
    \includegraphics[width=0.47\textwidth]{figs/MadelonNormalized_logistic_100e+02_PrimalLoss_Journal.eps}
    \caption{Logistic loss on the Madelon \citep{guyon2003design}. Left : $\ell_2$ regularization parameter $\mu = 10^{-2}$. Right : $\mu = 10^{2}$. Comparison of online RNA on primal-dual gradient methods with other first-order algorithms.}
    \label{fig:madelon_nonsym_logistic1}
\end{figure}